\documentclass[12pt,a4paper]{amsart}
\usepackage{amsmath,amsfonts,amssymb,amscd,amsthm}
\usepackage[shortlabels]{enumitem}

\usepackage{color, soul}
\definecolor{gr}{rgb}{0.7, 1, 0.7}
\definecolor{rr}{rgb}{1, 0.7, 0.7}

\usepackage{latexsym}
\usepackage{graphicx}

\theoremstyle{plain} 
\newtheorem{theorem}{Theorem}[section]
\newtheorem{lemma}[theorem]{Lemma}

\newtheorem{proposition}[theorem]{Proposition}

\theoremstyle{definition} 
\newtheorem{definition}[theorem]{Definition}
\theoremstyle{remark} 
\newtheorem{remark}[theorem]{Remark}

\renewcommand{\mathfrak}{\mathbf}

\renewcommand{\Re}{\,\mathrm{Re}\,}

\newcommand{\ignore}[1]{}

\newcommand{\bbC}{\mathbb{C}}
\newcommand{\bbH}{\mathbb{H}}

\newcommand{\bbN}{\mathbb{N}}
\newcommand{\bbR}{\mathbb{R}}

\newcommand{\bbD}{\mathbb{D}}

\newcommand{\tl}{\tilde}
\newcommand{\NN}{\mathbb N}

\newcommand{\cR}{\mathcal R}

\newcommand{\diam}{\operatorname{diam}}

\renewcommand{\mod}{\operatorname{mod}}

\title[Complex a priori bounds for Lorenz maps]{Complex a priori bounds for Lorenz maps}
\author{Denis Gaidashev}
\address{Uppsala University, Uppsala, Sweden}
\email{gaidash@math.uu.se}

\author{Igors Gorbovickis}
\address{Jacobs University, Bremen, Germany}
\email{i.gorbovickis@jacobs-university.de}

\subjclass[2010]{}
\date{\today}

\begin{document}
\begin{abstract}
We construct complex a-priori bounds for certain infinitely renormalizable Lorenz maps. As a corollary, we show that renormalization is a real-analytic operator on the corresponding space of Lorenz maps.
\end{abstract}
\maketitle

\section{Introduction}

Renormalization has been one of the central topics in modern one-dimensional real and complex dynamics. It was introduced in the works of Feigenbaum~\cite{Feigenbaum_78}, \cite{Feigenbaum_79} and Coullet and Tresser~\cite{Coullet_Tresser_78} in order to explain certain universality phenomena for families of unimodal maps and critical circle maps. Seminal works of Sullivan~\cite{Sull-qc}, \cite{deMelo_vanStrien} and Douady and Hubbard~\cite{DH85} put renormalization theory into the context of holomorphic dynamics. Renormalization theory of analytic unimodal maps of an interval was essentially completed in the works of McMullen~\cite{McM-ren2}, Lyubich~\cite{Lyubich-dynquad}, \cite{Lyubich-Hairiness}, \cite{Lyu-ae}, Graczyk and \'{S}wiatek~\cite{Graczyk_Swiatek_97}, and Levin and van Strien~\cite{Levin_vStr_98} while an analogous theory for analytic critical circle maps was developed in the series of works by de Faria~\cite{DeFar}, de Faria and de Melo~\cite{FM2}, and Yampolsky~\cite{Yamp-towers}, \cite{Ya3}, \cite{Ya4}.

The analyticity assumption of the maps plays an important role in the above mentioned works. 
Some further developments in~\cite{DeFaria_DeMelo_Pinto}, \cite{Guarino_deMelo}, \cite{Guarino_Martens_deMelo}, \cite{IG_Ya_unimodal}, \cite{IG_Ya} extend the above mentioned renormalization theories to certain classes of smooth maps, however, the proofs of these results still rely heavily on analytic methods. 

Motivated by a successful application of analytic methods in the renormalization theories of unimodal and critical circle maps, in this paper we introduce a complex analytic approach to the study of renormalization of Lorenz maps. 
Lorenz maps can be defined as orientation preserving maps of an interval with a single point of discontinuity. Such maps appear as factorized first return maps of a geometric Lorenz flow. In this paper we will study Lorenz maps that are real analytic outside of the point of discontinuity $c$ and have vanishing one-sided derivatives at $c$.

Renormalization of Lorenz maps was first considered by Martens and de Melo in~\cite{Martens_deMelo_2001}, where combinatorics of renormalizations was studied. Existence of a fixed point of renormalization was first proven by Winckler in~\cite{Winckler_2010} using computer assisted methods. In subsequent papers Winckler and Gaidashev~\cite{Gaidashev_Winckler_2012}, Martens and Winckler~\cite{Martens_Winckler_2014} and Gaidashev~\cite{Gaidashev_2012} constructed fixed points of Lorenz renormalization for wider classes of combinatorics.

Unlike in the case of analytic unimodal or critical circle maps, very little is currently known about convergence of renormalizations of Lorenz maps. On the contrary, Winckler and Martens have constructed examples of infinitely renormalizable Lorenz maps, whose renormalizations diverge~\cite{Martens_Winckler_2014b}. More specifically, they showed that these maps do not admit \textit{real a priori bounds}, that is, the sequence of renormalizations of any such map eventually leaves any $C^1$-compact subset in the space of all Lorenz maps.

In this paper we focus on the opposite case, when Lorenz maps admit real a priori bounds. Several classes of such maps were constructed in~\cite{Martens_Winckler_2014},~\cite{Gaidashev_2012}. Historically, one of the key steps in renormalization theory of analytic unimodal and critical circle maps was promotion of real a priori bounds to the so-called \textit{complex bounds}. The latter in particular imply that sufficiently high renormalizations have analytic extensions to some definite complex neighborhoods. The purpose of this paper is to prove complex bounds for analytic Lorenz maps of an arbitrary bounded (not necessarily monotone) combinatorial type admitting real a priori bounds (c.f., Theorem~\ref{main_theorem}). Our proof relies on the methods of~\cite{Ya1} and a careful analysis of the combinatorial properties of renormalizable Lorenz maps. As a corollary, we prove that a sufficiently high iterate of the renormalization is a real-symmetric analytic operator on an appropriate real-symmetric complex Banach manifold whose real slice consists of analytic Lorenz maps (c.f., Theorem~\ref{main_theorem2}). Analogous analytic operators for renormalization of unimodal and critical circle maps were recently used in~\cite{IG_Ya} and~\cite{IG_Ya_unimodal} to establish new results on convergence of renormalizations.

The structure of the paper is as follows: in Section~\ref{prelim_sec} we give all necessary definitions and state our main results. In Section~\ref{Complex_neib_sec} we collect the main complex analytic tools that are used in the proof of complex bounds, while in Section~\ref{dyn_intervals_sec} we study combinatorial properties of renormalizable Lorenz maps. Finally, Section~\ref{Main_proofs_sec} contains the proofs of our main results.

\section{Preliminaries and main results}\label{prelim_sec}

\subsection{Lorenz maps and their renormalization}

\begin{definition}
Let $\alpha>1$ be a real number and let $I\subset\bbR$ be a closed interval. A map $f\colon I\to I$ is an \textit{analytic unimodal map} with critical exponent $\alpha$, if there exists a point $c=c_f$ in the interior of $I$, such that 
\begin{enumerate}[(i)]
 \item $f$ is analytic on $I\setminus\{c\}$, and $f'(x)\neq 0$, for all $x\in I\setminus\{c\}$;
 \item in a neighborhood of the point $c$, the function $f$ can be reperesented as 
\begin{equation}\label{f_equals_psi_alpha_phi_equation}
f(x)=\psi(|\phi(x)|^\alpha),
\end{equation}
where $\phi$ is an affine map, $\phi(c)=0$, and $\psi$ is a conformal maps in some neighborhoods of $0$ respectively.
\end{enumerate}
\end{definition}

\begin{remark}
	One can consider a larger class of unimodal maps $f$ by letting the ``inner'' map $\phi$ be conformal in a neighborhood of $c$, however, it will be easy to check that subsequent renormalizations decrease the nonlinearity of the ``inner'' map, thus, bringing it arbitrarily close to the space of affine maps (see~\cite{IG_Ya_unimodal} for more details).
\end{remark}

\begin{definition}
 Let $I=[a,b]\subset\bbR$ be a closed interval and let $c\in I$ be an interior point of $I$. 
 A map 
 $f\colon I\setminus\{c\}\to I$ is called an \textit{analytic Lorenz map}, 
 if the following holds:
 \begin{enumerate}[(i)]
  \item $f(a)=a$, $f(b)=b$;
  \item $f$ can be represented in the form
 $$
 f(x)=
 \begin{cases}
  \hat f_-(x), & \text{ if } x\in[a,c)\\
  \hat f_+(x), & \text{ if } x\in(c,b],
 \end{cases}
 $$
 where $\hat f_+$ and $\hat f_-$ are analytic unimodal maps (on some intervals containing $(c,b]$ and $[a,c)$ respectively) with the same critical exponent 
 and with $\hat f_-'(c)=\hat f_+'(c)=0$;
 \end{enumerate}
\end{definition}
\begin{figure}[ht]
	\centerline{\includegraphics[width=10cm]{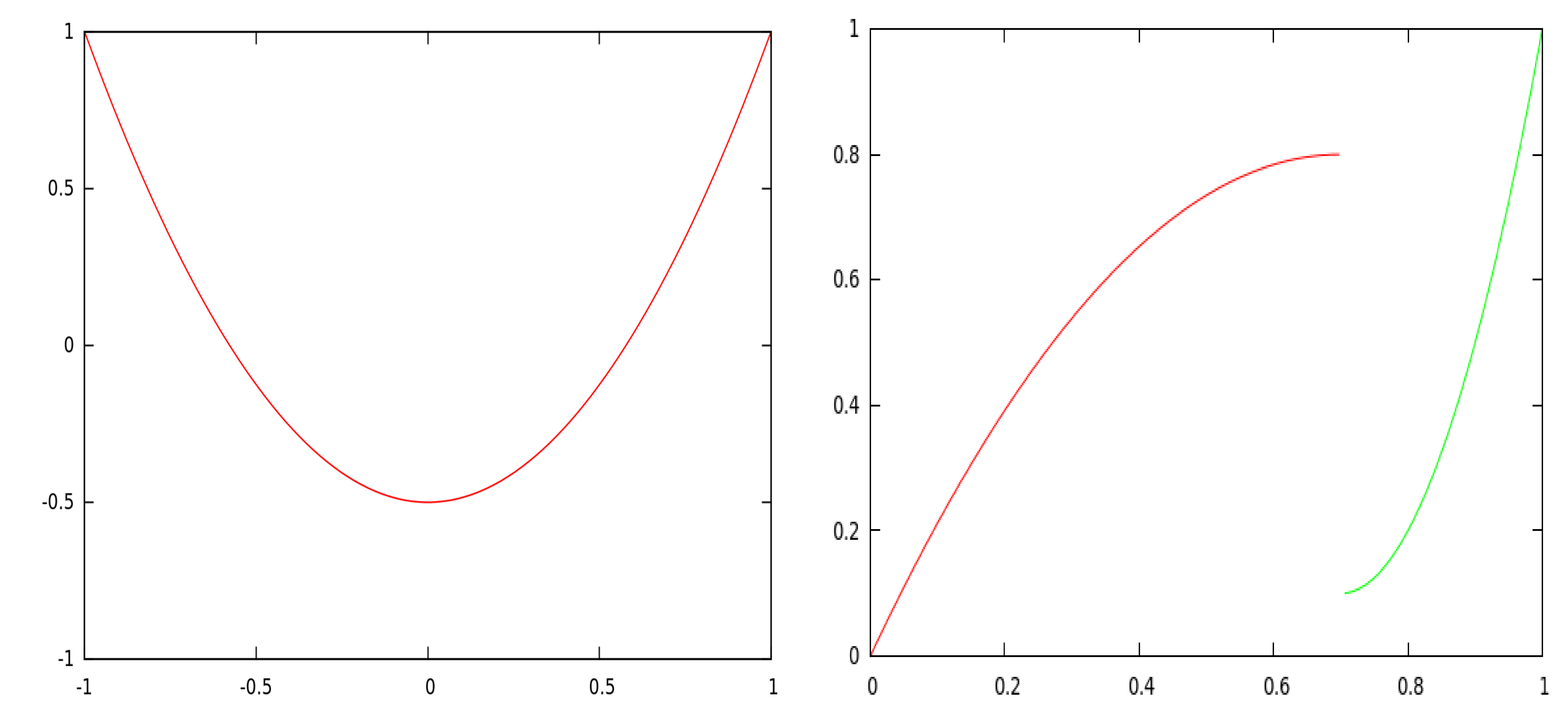}}
	\caption{\label{fig}A graph of a unimodal map on the left and a graph of a Lorenz map on the right.}
\end{figure}
The space af all analytic Lorenz maps, defined on $[0,1]$ and with a fixed critical exponent $\alpha>1$, will be denoted by $\mathfrak L$. (The critical exponent $\alpha>1$ will remain fixed and unchanged throughout the entire paper.) 

The maps $f_-=\hat f_-|_{[a,c]}$ and $f_+=\hat f_+|_{[c,b]}$ will be called respectively the \textit{left} and \textit{right branches} of a Lorenz map $f$.

A Lorenz map $f$ is \textit{nontrivial}, if $f_-(c)> c$ and $f_+(c)< c$. Otherwise, $f$ is called \textit{trivial} and has trivial dynamics. We will say that a Lorenz map $f$ is \textit{weakly nontrivial} if $f_-(c)\ge c$ and $f_+(c)\le c$.

 A (nontrivial) Lorenz map $f\in\mathfrak L$ is \textit{renormalizable}, if there exists a closed interval $C\subset(f_+(c),f_-(c))$, such that $c\in C$, and the first return map of the interior of $C$ is well defined and extends to a weakly nontrivial Lorenz map on $C$ without fixed points in the interior of $C\setminus\{c\}$.
Let $C$ be the maximal such interval. 
(It is easy to check that the maximal interval exists provided that $f$ is renormalizable. The condition on the absence of interior fixed points of the first return maps is crucial here.) We will call it \textit{the renormalization interval} of $f$. The \textit{renormalization} of $f$ is defined as 
$$
\cR f=A\circ R_C\circ A^{-1},
$$
where $R_C$ is the first return map of $f$ to $C$, and $A\colon C\to [0,1]$ is the unique orientation preserving affine homeomorphism between $C$ and $[0,1]$. In particular, $\cR f\in\mathfrak L$.

If $\cR f$ is also renormalizable, then we say that $f$ is \textit{twice renormalizable}. This way we define $n$ times renormalizable Lorenz maps, for all $n=1, 2, 3,\dots$, including $n=\infty$.

\subsection{Combinatorics of renormalizations}
If a Lorenz map $f$ is renormalizable with the renormalization interval $C$, then the critical point $c$ splits $C$ into two 
intervals $C_-=\{x\in C\mid x< c\}$ and $C_+=\{x\in C\mid x> c\}$, and the first return map $R_C$ has the form
$$
R_C(x)=
\begin{cases}
 f^{m_-}(x), & \text{ if } x\in C_-\\
 f^{m_+}(x), & \text{ if } x\in C_+,
\end{cases}
$$
for some positive integers $m_-,m_+\in\bbN$. 
Since $R_C$ is the first return map of the interior of $C$, the interiors of the intervals 
$$
C_-, f(C_-), f^2(C_-),\dots, f^{m_--1}(C_-)
$$
are pairwise disjoint (otherwise, different points of $C_-$ would have different return times) and their relative order determines a permutation $\theta_-(f)$ of $\{0,1,\dots,m_--1\}$. Similarly, the relative order of the intervals 
$$
C_+, f(C_+), f^2(C_+),\dots, f^{m_+-1}(C_+)
$$
determines a permutation $\theta_+(f)$ of $\{0,1,\dots,m_+-1\}$. We define a permutation $\theta(f)$ as the direct product (ordered pair) of permutations $\theta(f)=(\theta_-(f),\theta_+(f))$.
We say that a permutation $\theta=(\theta_-,\theta_+)$ is \textit{Lorenz}, if there exists a renormalizable Lorenz map $f$, such that $\theta=\theta(f)$. The set of all Lorenz permutations will be denoted by $\mathbf P$.

For $n\in\bbN\cup\{\infty\}$ and a subset $\Theta\subset\mathbf P$, let $\mathcal S_\Theta^n\subset\mathfrak L$ be the set of all $n$ times renormalizable Lorenz maps $f$, such that $\theta(\cR^j f)\in\Theta$, for all $j=0,1,2,\dots,n-1$. For $f\in\mathcal S_{\mathbf P}^n$, let $\rho_n(f)$ be the finite or infinite sequence of permutations $(\theta(f),\theta(\cR f),\theta(\cR^2 f),\dots, \theta(\cR^{n-1} f))\subset\mathbf P^n$. For further convenience we also define $\mathcal S_\Theta^0:=\mathfrak L$.  

We say that two infinitely renormalizable Lorenz maps $f$ and $g$ are of the same combinatorial type, if $\rho_\infty(f)=\rho_\infty(g)$.

\subsection{Real bounds}

For any Lorenz map $f=(f_-,f_+)\in\mathfrak L$ with critical point $c_f\in (0,1)$, there exist two $C^\infty$-smooth functions $\eta_f^-\colon [0,c_f^\alpha]\to\bbR$ and $\eta_f^+\colon [0,(1-c_f)^\alpha]\to\bbR$, such that the functions $f_\pm$ can be represented as
$$
f_\pm(x) = \eta_f^\pm(|x-c_f|^\alpha).
$$

\begin{definition}\label{real_bounds_def}
For any pair of real numbers $\delta,\Delta>0$, we say that a Lorenz map $f\in\mathfrak L$ has real $(\delta,\Delta)$-bounds of level $n\in\{0\}\cup\bbN$, if $f\in\mathcal S_{\mathfrak P}^n$, and for all $k=0,1,\dots,n$, the maps $f_k=\cR^k f$ satisfy the following conditions:
$$
\delta \le c_{f_k} \le 1-\delta,\qquad\text{and}\qquad \left| \frac{(\eta_{f_k}^\pm)''}{(\eta_{f_k}^\pm)'}\right| \le \Delta.
$$
We say that $f\in\mathfrak L$ has real $(\delta,\Delta)$-bounds, if it has real $(\delta,\Delta)$-bounds of arbitrarily high level.
\end{definition}

\begin{lemma}\label{derivative_bounds_lemma}
For any pair of real numbers $\delta,\Delta>0$, there exist positive real constants $K_1>1$ and $K_2>0$, such that for any nontrivial Lorenz map $f\in\mathfrak L$ with real $(\delta,\Delta)$-bounds of level $0$, we have
\begin{equation}\label{eta_f_bounds}
\frac{1}{K_1} < |(\eta_{f}^\pm)'| < K_1 \qquad\text{and}\qquad |(\eta_{f}^\pm)''|<K_2.
\end{equation}
\end{lemma}
\begin{proof}
The condition 
$\left| {(\eta_{f}^\pm)''}/{(\eta_{f}^\pm)'}\right| \le \Delta$ implies that the maps $\eta_{f}^\pm$ have bounded distortion. The condition $\delta\le c_f\le 1-\delta$ together with nontriviality of $f$ implies that the absolute values of the derivatives $|(\eta_{f}^\pm)'|$ are bounded from below by a positive constant. Now boundedness of distortion implies that they are also bounded by some constant from above. Then the inequality on the second derivatives follows.
\end{proof}

The existence of compact sets of Lorenz maps with real bounds of level $\infty$ (using real techniques) has been shown for several specific combinatorial types. Real bounds for {\it monotone} combinatorial types, that is the combinatorics for which the intervals $f(C_\pm), f^2(C_\pm), \ldots,  f^{m_\pm-1}(C_\pm)$  belong to the left ($+$) or right ($-$) component of $[0,1] \setminus c$, have been proved in \cite{Martens_Winckler_2014} for the following return times $m_\pm$: 
$$\alpha+\sigma \le m_- \le 2\alpha-\sigma, \quad b_0 \le m_+ -1 \le (1+(\sigma/\alpha)^2-\beta) b_0,$$
where $\alpha >1$, $\sigma \in (0,1)$, $\beta \in (0,(\sigma/\alpha)^2)$, and $b_0 \in \NN$.  
On the other hand, \cite{Gaidashev_2012} demonstrates existence of real bounds for short return times 
$$(\ln 2 / \ln \alpha +1) \alpha <  m_\pm  <  2 \alpha.$$

\subsection{Analytic extensions of Lorenz maps}

Given a Lorenz map $f\in\mathfrak L$, it follows from the power law~(\ref{f_equals_psi_alpha_phi_equation}) that $f_-$ and $f_+$ have univalent analytic extensions to some corresponding real-symmetric domains $U_-,U_+\subset\bbC$, such that the sets $f_-(U_-)\cup\bbR$ and $f_+(U_+)\cup\bbR$ contain some open neighborhoods of $f_-([a,c])$ and $f_+([c,b])$ respectively. 

\begin{definition}\label{L_r_def}
	For a positive real number $r>0$, let $\mathfrak L_r\subset\mathfrak L$ denote the set of Lorenz maps $f$, for which the following holds:
	\begin{enumerate}[(i)]
		\item there exist domains $U_-,U_+$ with the above properties and such that the sets $f_-(U_-)\cup\bbR$, $f_+(U_+)\cup\bbR$ contain open $r$-neighborhoods of $f_-([a,c])$ and $f_+([c,b])$ respectively;
		\item\label{L_r_def_item2}
		 there exist neighborhoods $V_-,V_+\subset\bbC$ of the critical point $c$, in which the power laws~(\ref{f_equals_psi_alpha_phi_equation}) for $f_-$ and $f_+$ respectively hold, and the closures of the sets $f_-(U_-\cap V_-)$ and $f_+(U_+\cap V_+)$ contain disks of radius $r$, centered at $f_-(c)$ and $f_+(c)$ respectively.
	\end{enumerate}
\end{definition}

\subsection{Power-like Lorenz maps}

\begin{definition}
	Let $U_+,U_-, D\subset\bbC$ be three simply connected real-symmetric domains such that $U_\pm\Subset D$, and $I_+:=U_+\cap\bbR$ and $I_-:=U_-\cap\bbR$ are two disjoint open intervals with a common endpoint $c\in\bbR$. A pair of maps 
	$$
	f_\pm\colon U_\pm\to D
	$$
	is called a \textit{power-like Lorenz map}, if $f_\pm$ are conformal diffeomorphisms of $U_\pm$ onto $(D\setminus\bbR)\cup f_\pm(I_\pm)$ and the maps $f_+$ and $f_-$ restrict respectively to the left and right branches of a Lorenz map $f\in \mathfrak L$.	
\end{definition}

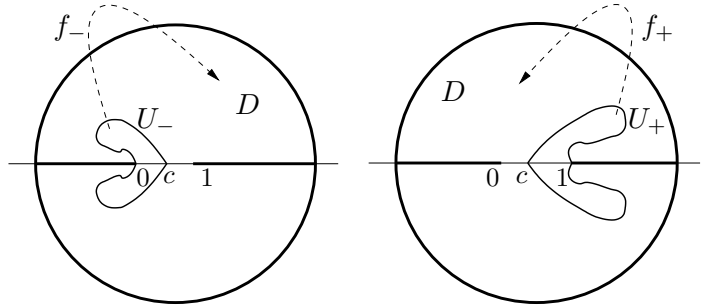
\begin{figure}[ht]
\centering
\vspace{2mm}       
\begin{tabular}{ c c }
\includegraphics[height=4cm]{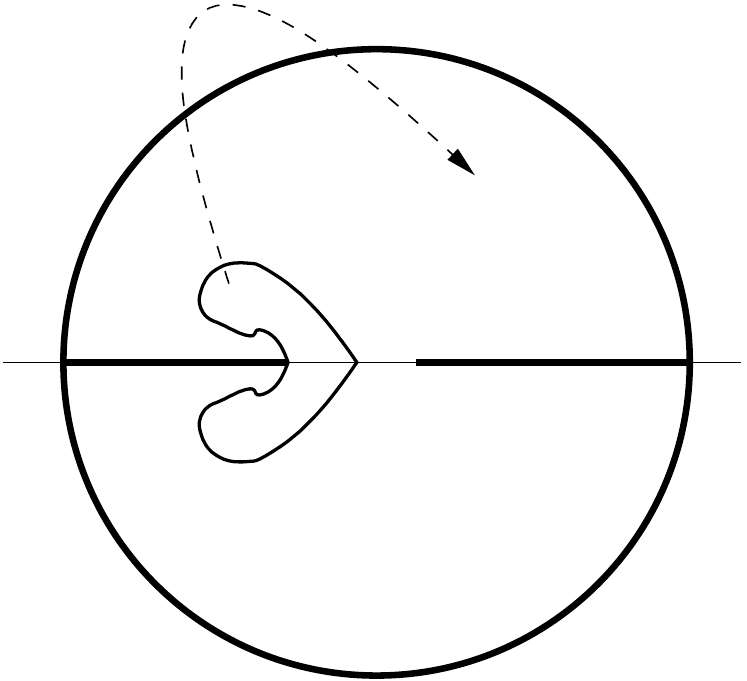} &  \includegraphics[height=4cm]{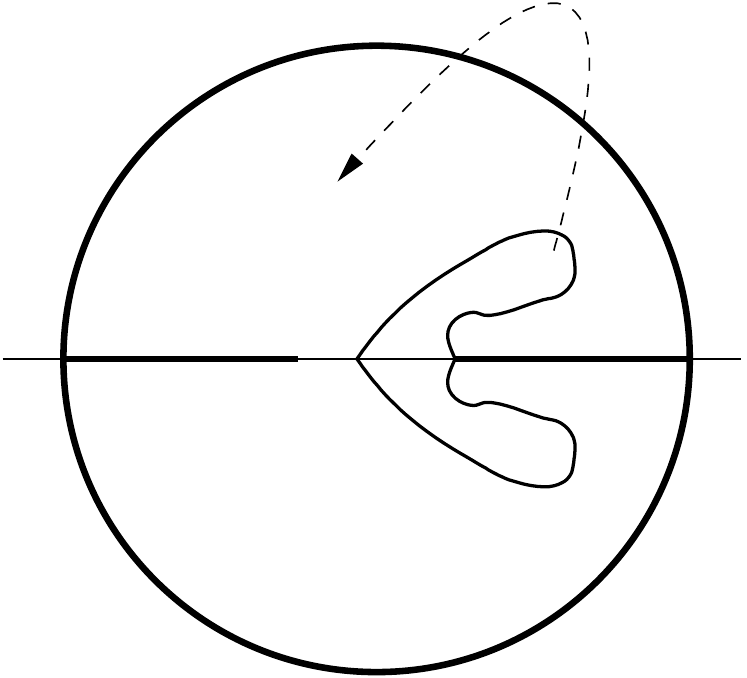} 
\end{tabular}
\put(-28,50){\small $f_+$}
\put(-103,25){\small $D$}
\put(-33,15){\small $U_+$}
\put(-75,-6){\footnotesize $c$}
\put(-86,-7){\footnotesize $0$}
\put(-60,-7){\footnotesize $1$}
\put(-248,50){\small $f_-$}
\put(-180,20){\small $D$}
\put(-217,15){\small $U_-$}
\put(-207,-6){\footnotesize $c$}
\put(-193,-7){\footnotesize $1$}
\put(-217,-7){\footnotesize $0$}
\caption{A power-like Lorenz map}
\label{fig:comp_bounds_fig}       
\end{figure}

\begin{definition}
	For a simply connected domain $U\subset\bbC$ and a set $X\Subset U$, let
	$
	\mod(X,U)
	$
	denote the supremum of the moduli of all annuli $A\subset U\setminus\overline X$ that separate $\partial U$ from $\overline X$. 
\end{definition}

\begin{definition}\label{H_nu_def}
	For $\nu\in (0,1/2)$, let $\mathbf H(\nu)$ denote the set of all power-like Lorenz maps $f=(f_-,f_+)$, such that the following conditions are simultaneously satisfied:
	\begin{enumerate}[(i)]
	\item\label{H_nu_item1} $U_\pm\Subset f_\pm(U_\pm)$ and $\mod(U_\pm, f_\pm(U_\pm))\ge \nu$;
	\item\label{H_nu_item2} $\diam(D)\le 1/\nu$;
	\item\label{H_nu_item3} $c\in[\nu,1-\nu]$;
	\item\label{H_nu_item4} $f\in\mathfrak L_\nu$.
	\end{enumerate}
\end{definition}

\subsection{Statement of results}

We are now ready to give precise statements of our main results.

\begin{theorem}[\textbf{Complex bounds}]\label{main_theorem}
For any pair of real numbers $\delta,\Delta>0$ and a finite set $\Theta\subset\mathfrak P$, there exists a real number $\nu>0$, such that for any real $r>0$, 
there exists $n_0=n_0(r,\delta,\Delta,\Theta)\in\bbN$  with the property that for any $n \ge n_0$ and $f\in\mathfrak L_r \cap \mathcal S_\Theta^{n+1}$ with real $(\delta,\Delta)$-bounds of level $n$, the renormalization $\cR^nf$ extends to a power-like Lorenz map from~$\mathbf H(\nu)$.
\end{theorem}

A key step in the proof of Theorem~\ref{main_theorem} is Lemma~\ref{main_main_lemma}, stated in the beginning of Section~\ref{Main_proofs_sec}. As a corollary from complex bounds, we deduce the following theorem:

\begin{theorem}[\textbf{Analyticity of renormalization}]\label{main_theorem2}
For any pair of real numbers $\delta,\Delta>0$ and a finite set $\Theta\subset\mathfrak P$, there exist a positive integer $N>0$, a real-symmetric analytic Banach manifold $\mathbf M$ whose real slice $\mathbf M^\bbR$ consists of analytic Lorenz maps, and an open set $\mathcal O\subset\mathbf M$, such that the following holds:
\begin{enumerate}
	\item\label{Analyt_prop1} the $N$-th iterate $\cR^N\colon \mathcal O\to \mathbf M$ is defined and is a real-symmetric analytic operator on $\mathcal O$;
	
	\item\label{Analyt_prop2} for every $f\in\mathcal S_\Theta^\infty$ with real $(\delta,\Delta)$-bounds, there exists a positive integer $M>0$, such that for every $n\ge M$, we have $\cR^n f\in\mathcal O$. 
\end{enumerate}
\end{theorem}


\section{Complex neighborhoods}\label{Complex_neib_sec}
In this section we collect the tools that we later use to control the behavior of inverse branches of Lorenz maps in complex neighborhoods of real intervals. One of the important ingredients is a strengthened version of Lemma~3.3 from~\cite{FM2}. 
The results of this section are quite general and can be applied to arbitrary real-symmetric analytic maps.

For an open interval $J\subset\bbR$, we define the domain $\bbC(J)\subset\bbC$ as $\bbC(J)=(\bbC\setminus\bbR)\cup J$. We let $\bbD(J)$ denote the intersection of $\bbC(J)$ with the open unit disk, centered at the midpoint of $J$. 
Given an open interval $J\subset\bbR$ and a real number $t>0$, we denote by $D_t(J)\subset \bbC$ the set of all $z\in\bbC$ that view $J$ under angle $\ge 2\arctan(t)$. Each $D_t(J)$ is a hyperbolic neighborhood of $J$ of some radius $r=r(t)$ in $\bbC(J)$. In other words, $D_t(J)$ is the set of all points in $\bbC(J)$, whose hyperbolic distance to $J$ is less than $r$. A version of the Schwarz lemma for such domains can be formulates as follows (c.f.~\cite{Ya1}):
\begin{lemma}[\textbf{Schwarz Lemma}]\label{Schwarz_lemma}
Consider an open interval $J\subset\bbR$ and let $\phi\colon\bbC(J)\to\bbC(J)$ be an analytic map, such that $\phi(J)\subset J$. Then for any $t>0$, we have $\phi(D_t(J))\subset D_{t}(J)$.
\end{lemma}

\begin{figure}[th]
	\centering
	\vspace{2mm}       
	\includegraphics[height=3cm]{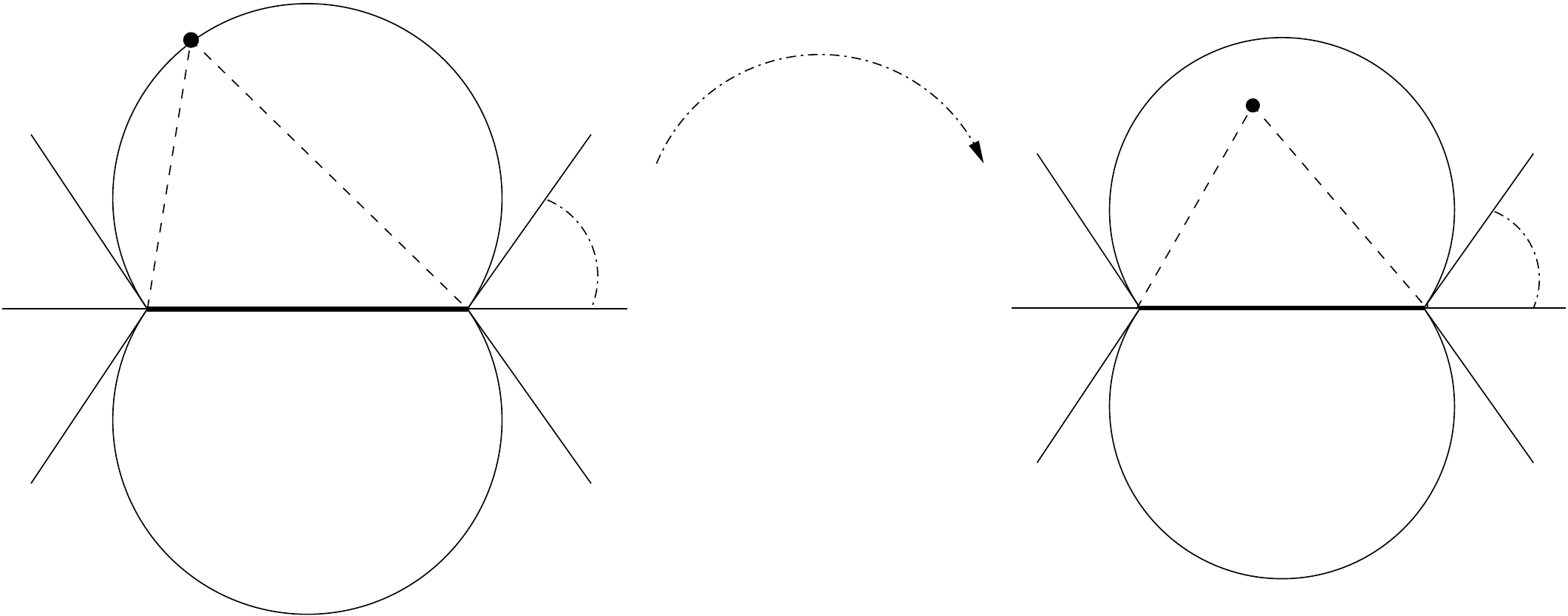} 
	\caption{Hyperbolic neighborhoods in Schwarz Lemma.}
	\label{fig:schwarz}       
\end{figure}

Let $|J|$ stand for the length of $J$. If the map $\phi$ in the above lemma is defined only in the slit disk $\bbD(J)$, then the following analogous statement holds (compare with Lemma~4.4 from~\cite{Ya_bicritical}):

\begin{lemma}\label{Generalized_S_lemma}
 Consider an open interval $J\subset\bbR$ and set $a=|J|/2$. Assume, $a<1$ and let $\phi\colon\bbD(J)\to\bbC(J)$ be an analytic map, such that $\phi(J)\subset J$. Then for any $t>0$, such that $D_t(J)\subset\bbD(J)$, we have
 $$
 \phi(D_t(J))\subset D_{\tilde t}(J),\qquad\text{where}\qquad \tilde t=\frac{t^2-a^2}{t(1+a^2)}.
 $$
\end{lemma}
\begin{proof}
 Without loss of generality we can assume that $J=(-a,a)$. A direct computation shows that the map
 $$
 F(z)=\frac{(a^2+1)z}{z^2+1}
 $$
 is a real-symmetric conformal diffeomorphism from $\bbD(J)$ to $\bbC(J)$. Furthermore, $F(a)=a$ and $F(-a)=-a$. For $z\in\bbH$, the argument of the complex number
 $$
 R(z)=\frac{z-a}{z+a}
 $$
 is equal to the angle under which the interval $J$ is viewed from $z$. For $z\in\bbH\cap\bbD$, a direct computation yelds
 $$
 R(F(z))=-R(z)R(a^2z).
 $$
 The minimum
 $$
 \min_{z\in \partial (D_t(J)\cap\bbH)}\arg(R(a^2z))
 $$
 must be achieved at a unique point $z_t\in \partial (D_t(J)\cap\bbH)$ that is the point of tangency between the circular arc 
 $$
 a^2\partial (D_t(J)\cap\bbH) := \{a^2z \mid z\in \partial (D_t(J)\cap\bbH)\}
 $$
 and some circle through the points $a$ and $-a$. Due to the vertical symmetry, the point $z_t$ must lie on the imaginary axis, so $z_t=ai/t$. Now it follows that
 $$
 \min_{z\in \partial (D_t(J)\cap\bbH)} \arg [R(F(z))]
 $$
 is achieved at the same point $z_t$. 
 Thus, the boundary of the smallest set $D_{\tilde t}(J)$ containing $F(D_t(J))$, passes through the point $F(z_t)$. According to a direct computation, $F(z_t)$ lies on the imaginary axis, and the expression for $\tilde t$ easily follows. Now the lemma follows directly from Lemma~\ref{Schwarz_lemma}. 
\end{proof}

We note that the statement of our Lemma~\ref{Generalized_S_lemma} is stronger than Lemma~3.3 from~\cite{FM2}. First of all, Lemma~\ref{Generalized_S_lemma} provides a precise (rather than just asymptotic) expression for $\tl t$, but more importantly, it does not require $\phi$ to be univalent and analytic in the whole disk $\bbD$. The later enables us to prove the following:

\begin{lemma}\label{definite_neighb_lemma}
For positive real numbers $L, \Delta t>0$, such that $0<\Delta t<1$, there exists $d>0$, for which the following property holds:
let $I_1,\dots,I_{n+1}\subset\bbR$ be a finite family of 
intervals and for each $k=1,\dots,n$, let $f_k\colon \bbD(I_k)\to\bbC(I_{k+1})$ be a real-symmetric analytic map. If 
\begin{equation}\label{Interval_condidtions}
\sum_{k=1}^n |I_k|\le L\qquad \text{and}\qquad \max_{1\le k\le n}|I_k|\le d,
\end{equation}
then for any $t\in\bbR$ and $k\in\bbN$, such that $\Delta t<t<1$ and $k\le n$, the composition $F_k=f_k\circ f_{k-1}\circ\dots\circ f_1$ is well defined over $D_t(I_1)$ and maps this neighborhood to $D_{t-\Delta t}(I_{k+1})$.
\end{lemma}
\begin{proof}
Fix a positive real constant $\delta<1$.
First we observe that for any sufficiently small $d>0$, there exists a constant $0<K_d<1$, such that $K_d\to 1$ as $d\to 0$, and if intervals $I_1,\dots I_n$ satisfy~(\ref{Interval_condidtions}), then
\begin{equation}\label{decrease_proportion_ineq}
\prod_{j=1}^n \left(1-\left(\frac{|I_j|}{2}\right)^{1+\delta}\right)\ge K_d.
\end{equation}
Indeed, for all sufficiently small $d>0$ we have
\begin{multline*}
\log \left[ \prod_{j=1}^n \left(1-\left(\frac{|I_j|}{2}\right)^{1+\delta}\right) \right]
=\sum_{j=1}^n \log \left(1-\left(\frac{|I_j|}{2}\right)^{1+\delta}\right) \ge \\
\ge -2\sum_{j=1}^n \left(\frac{|I_j|}{2}\right)^{1+\delta}
\ge -2^{-\delta} d^\delta \sum_{j=1}^n |I_j|
\ge -2^{-\delta} d^\delta L,
\end{multline*}
and $2^{-\delta} d^\delta L \to 0$ as $d\to 0$, which implies~(\ref{decrease_proportion_ineq}).

Next, we observe that for any sufficiently small $d>0$, the condition $|I_j|\le d$ implies
\begin{multline}\label{two_to_delta_ineq}
\left(1-\left(\frac{|I_j|}{2}\right)^{2}\right)\left(1-\left(\frac{|I_j|}{2K_d\Delta t}\right)^{2}\right) \ge \\
\ge 1 - \left(\frac{|I_j|}{2}\right)^{2}\left( 1+\frac{1}{K_d^2\Delta t^2}\right)
\ge 1 - \left(\frac{|I_j|}{2}\right)^{1+\delta}.
\end{multline}

Finally, let $t_1=t$, and for $1\le k\le n$, we define
\begin{equation}\label{t_k_eq}
t_{k+1}=\frac{t_k^2-|I_k|^2/4}{t_k(1+|I_k|^2/4)}.
\end{equation}
According to Lemma~\ref{Generalized_S_lemma}, if $t_j>0$ and $D_{t_j}(I_j)\subset\bbD(I_j)$, for all $j=1,\dots, k$, then the composition $F_k$ is defined over $D_{t_1}(I_1)$, and $F_k(D_{t_1}(I_1))\subset D_{t_{k+1}}(I_{k+1})$. Thus, in order to prove the lemma, it is enough to show that for all sufficiently small $d>0$, the conditions~(\ref{Interval_condidtions}) imply that
$$
t_{k+1}\ge K_d t_1, \qquad\text{for all }k=1,\dots,n.
$$

Under the same conditions we will use induction to prove a stronger inequality
$$
t_{k+1}\ge t_1 \prod_{j=1}^k \left(1-\left(\frac{|I_j|}{2}\right)^{1+\delta}\right).
$$

\textit{Base case}: it follows from~(\ref{t_k_eq}),~(\ref{two_to_delta_ineq}) and the condition $t_1>\Delta t$ that 
\begin{multline*}
t_{2} =\frac{t_1-|I_1|^2/(4t_1)}{(1+|I_1|^2/4)} \ge t_1\left(1-\left(\frac{|I_1|}{2}\right)^2\right) \left(1-\left(\frac{|I_1|}{2t_1}\right)^2\right) \ge \\
\ge t_1\left( 1 - \left(\frac{|I_1|}{2}\right)^{1+\delta} \right). 
\end{multline*}

\textit{Induction step}: again, it follows from from~(\ref{t_k_eq}),~(\ref{two_to_delta_ineq}) and the induction hypothesis that
\begin{multline*}
t_{k+1} =\frac{t_k-|I_k|^2/(4t_k)}{(1+|I_k|^2/4)} \ge t_k\left(1-\left(\frac{|I_k|}{2}\right)^2\right) \left(1-\left(\frac{|I_k|}{2t_k}\right)^2\right) \ge \\
\ge t_k\left(1-\left(\frac{|I_k|}{2}\right)^2\right) \left(1-\left(\frac{|I_k|}{2K_d\Delta t}\right)^2\right) \ge \\
\ge t_k\left( 1 - \left(\frac{|I_k|}{2}\right)^{1+\delta} \right) \ge
t_1 \prod_{j=1}^k \left(1-\left(\frac{|I_j|}{2}\right)^{1+\delta}\right).
\end{multline*}
\end{proof}

Let $\phi_\alpha\colon\bbC\setminus\bbR_{<0}\to\bbC$ be the branch of the map $z\mapsto z^{1/\alpha}$ preserving $\bbR_{>0}$. Fix a positive real number $\sigma<\cot(\frac{\pi}{2\alpha})$. The following lemma is analogous to Lemma~2.2 from~\cite{Ya1}. The proof is left to the reader.

\begin{lemma}[\textbf{Root of degree $\alpha$}]\label{root_lemma}
	Let $K>0$ and $M\in(0,1)$ be positive real numbers.
	Then for any real numbers $a,t, c\in\bbR$, such that $t>0$, $a\in(0,K)$, and $c\in[0,M]$, 
	there exists a real number $\tl t=\tl t(K,M,t,\alpha)>0$, such that 
	$$
	\phi_\alpha(D_t((-a,1))\setminus (-a,0]) \subset D_\sigma((0,c)) \cup D_{\tl t}((c,1))\subset D_{\tl t}((0,1)).
	$$
\end{lemma}

Lemma~\ref{root_lemma} can also be reformulated for maps that are distorted roots of degree~$\alpha$, provided that there is some control of the distortion. A precise statement is given in the next lemma.
\begin{definition}
	Let $D\subset\bbC$ be a real-symmetric domain containing the interval $[0,1]$. For any real $\mu>0$, we will say that a function $f\colon D\setminus\bbR_{<0}\to \bbC$ is a distorted root of degree $\alpha$ on $D$ with modulus $\mu$, if $f(1)=1$ and $f$ can be represented as
	$$
	f=g\circ\phi_\alpha\circ h,
	$$
	where $h\colon D\to\bbC$ and $g\colon \phi_\alpha(h(D)) \to \bbC$ are conformal maps that fix the origin and can be extended to conformal maps of some domains $U$ and $V$ respectively, such that $\mod(D, U)\ge \mu$ and $\mod(\phi_\alpha(h(D)), V)\ge \mu$.
\end{definition}

\begin{lemma}[\textbf{Distorted root of degree $\alpha$}]\label{root_lemma_bounded_distortion}
	Let $\mu, K>0$ and $M\in(0,1)$ be positive real numbers.
	Then for any real numbers $a,t, c\in\bbR$, such that $t>0$, $a\in(0,K)$, and $c\in[0,M]$, 
	there exists a real number $\tl t=\tl t(\mu, K,M,t,\alpha)>0$, such that for any map $f\colon D_t((-a,1))\setminus (-a,0] \to \bbC$ that is a distorted root of degree $\alpha$ on $D_t((-a,1))$ with modulus $\mu$, we have the inclusion
	\begin{equation}\label{D_t_inclusion_equation}
	f(D_t((-a,1))\setminus (-a,0]) \subset D_\sigma((0,c)) \cup D_{\tl t}((c,1))\subset D_{\tl t}((0,1)).
	\end{equation}
\end{lemma}
\begin{proof}
It follows from Lemma~\ref{root_lemma} that for each particular function $f$ as above, one can choose a parameter $\tl t$ that satisfies~(\ref{D_t_inclusion_equation}). Furthermore, the parameter $\tl t$ can be chosen to depend continuously on $f$ in open-compact topology. Finally, according to the Koebe Distortion Theorem, the set of all functions $f$ that are distorted roots of degree $\alpha$ on $D_t((-a,1))$ with modulus $\mu$, is a normal family. Hence by the standard compactness argument, there exists a real number $\tl t$ that satisfies~(\ref{D_t_inclusion_equation}), simultaneously for all maps $f$ as above.
\end{proof}

\section{Dynamical intervals}\label{dyn_intervals_sec}

\subsection{Prerenormalization}
If $f\in\mathfrak L$ is an $n$ times renormalizable Lorenz map, then $\cR^n f$ can be represented as a rescaled first return map of $f$ to some closed interval $C_n$, such that $c\in C_n$. We denote this first return map by $p\cR^n f$ -- the $n$-th prerenormalization of $f$. The critical point $c$ splits $C_n$ into two 
intervals 
$$
C_{n-}=\{x\in C_n\mid x< c\}\qquad\text{and}\qquad C_{n+}=\{x\in C_n\mid x> c\}. 
$$
These intervals will be called the \textit{left} and \textit{right domains of the $n$-th prerenormalization} $p\cR^n f$. 
The restrictions of $p\cR^n f$ to $C_{n-}$ and $C_{n+}$ will be denoted by $p\cR^n f_-$ and $p\cR^n f_+$ respectively. Each of them is an analytic homeomorphisms of $C_{n-}$ and $C_{n+}$ respectively and can be represented as 
$$
p\cR^n f_-=f^{m_n^-},\qquad\text{and}\qquad p\cR^n f_+=f^{m_n^+},
$$
for some positive integers $m_n^-,m_n^+\in\bbN$. Each $f$ in these compositions is either the left or the right branch of $f$. 
Further we will always assume that the maps $p\cR^n f_-$ and $p\cR^n f_+$ are defined on the closures of the intervals $C_{n-}$ and $C_{n+}$ respectively by continuous extensions:
$$
p\cR^n f_-(c):= \lim_{x\to c-} p\cR^n f_-(x) \qquad\text{and}\qquad p\cR^n f_+(c):= \lim_{x\to c+} p\cR^n f_+(x).
$$

\subsection{Combinatorial properties of renormalizations}
In this subsection we prove some combinatorial properties of renormalizable Lorenz maps.
We will state these properties for the class of analytic Lorenz maps from $\mathfrak L$, however the statements and their proofs remain the same if instead of the maps from $\mathfrak L$ one considers the so-called \textit{topological Lorenz maps}, i.e., the Lorenz maps $f=(f_-,f_+)$, where $f_-$ and $f_+$ instead of being restrictions of analytic unimodal maps, are assumed to be just homeomorphisms to their images.

\subsubsection{Homeomorphic extensions of prerenormalizations}
Let $f\in\mathbf L$ be an $n$-times renormalizable Lorenz map. For each $k=1,\dots,n$, let $L_{k-},L_{k+}\subset (0,1)$ be the maximal open intervals, such that $C_{k\pm}\subset L_{k\pm}$, and $p\cR^k f_\pm$, viewed as an appropriate composition of the maps $f_\pm$, is a homeomorphism from $L_{k\pm}$ onto its image. We note that since $p\cR^k f_\pm$ is the first return map of $C_{k\pm}$ to $C_k$, it follows that the sets $C_{k\mp}\cap L_{k\pm}$ have no interior points, hence $C_{k\pm}\neq L_{k\pm}$. We define $L_k=L_{k-}\cup L_{k+}\cup\{c\}$.

\begin{lemma}\label{L_n_combinatorial_lemma}
Assume that for some $n\in\bbN$, a map $f\in\mathfrak L$ is $(n+1)$-times renormalizable. Then
$$
L_{n\pm}\Subset p\cR^nf_\pm(L_{n\pm}).
$$
Furthermore, if $A_{n\pm}$ and $B_{n\pm}$ are two connected components of $p\cR^nf_\pm(L_{n\pm})\setminus L_{n\pm}$, such that $c\in\partial A_{n\pm}$, then
$$
C_{(n+1)\mp}\subset A_{n\pm} 
$$
and $B_{n\pm}$ contains an interval from the finite orbit of $C_{(n+1)\mp}$ under the map $p\cR^{n-1}f$ before its first return to $C_n$.
\end{lemma}
\begin{proof}
Since the map $p\cR^nf$ is renormalizable, it is a nontrivial Lorenz map, hence $A_{n\pm}$ is a non-degenerate interval, such that $C_{(n+1)\mp}\subset A_{n\pm}$.

Now we will show that $B_{n\pm}$ is nonempty. 
Consider the interval $X_\pm=L_{n\pm}\setminus\overline{C_{n\pm}}$. Since $L_{n\pm}$ is the maximal interval on which $p\cR^nf_\pm$ is a homeomorphism, there exists a homeomorphic image $Y_\pm\subset C_{(n-1)\mp}$ of $X_\pm$ under some iterate of $p\cR^{n-1}f$, such that $c\in\partial Y_\pm$. Then we have $C_{n\mp}\subset Y_\pm$, since otherwise the orbit of $C_{n\pm}$ under the map $p\cR^{n-1}f$ would have common interior points with the interval $C_n$ before returning to $p\cR^nf(C_{n\pm})$. 
If the set $B_{n\pm}$ is empty, then $p\cR^nf_\pm$ maps $X_\pm$ homeomorphically into  itself, hence the iterates of $Y_\pm$ under the dynamics of $p\cR^{n-1}f$ never have the critical point $c$ in their interiors, so $p\cR^{n}f$ is a trivial map and cannot be renormalizable, which is a contradiction. 

Now let $k_\pm, m_\pm\in\bbN$ be such that $(p\cR^{n-1}f)^{\circ k_\pm}$ maps $Y_\pm$ inside $p\cR^nf_\pm(L_{n\pm})$ and $(p\cR^{n-1}f)^{\circ m_\pm}|_{C_{n\mp}}=p\cR^nf_\mp$. First, we observe that since $C_{n\mp}\subset Y_\pm$ and $(p\cR^{n-1}f)^{\circ k_\pm}$ maps $Y_\pm$ homeomorphically onto its image, we have $m_\pm>k_\pm$. 
Next, we notice that if $s_\pm\in\bbN$ is such that $(p\cR^{n-1}f)^{\circ s_\pm}(X_\pm)=Y_\pm$, then $m_\pm-k_\pm\le s_\pm$, since the renormalization $p\cR^nf$ is defined as the first return map to $C_n$. Now we observe that $B_{n\pm}\subset (p\cR^{n-1}f)^{\circ k_\pm}(Y_\pm)$, and since $p\cR^nf$ is a nontrivial Lorenz map, the interval $(p\cR^{n-1}f)^{\circ (m_\pm-k_\pm)}(B_{n\pm})$ contains the critical point $c$ either on the boundary (if $s_\pm=m_\pm-k_\pm$) or in the interior (if $s_\pm<m_\pm-k_\pm$).

Finally, the closed interval $\overline{p\cR^nf(C_{(n+1)_\mp})}$ does not contain the critical point $c$, hence $(p\cR^{n-1}f)^{\circ k_\pm}(C_{(n+1)_\mp})\subset B_{n\pm}$, which completes the proof.
\end{proof}

\subsubsection{Homeomorphic extensions of $f_\pm\circ(p\cR^nf_\pm)^{-1}$}

Let $f\in\mathfrak L$ be an $n$-times renormalizable Lorenz map. We observe that the map $f_\pm\circ (p\cR^n f_\pm)^{-1}$ on the interval $p\cR^n f_\pm(C_{n\pm})$ can be represented as the composition of $m_n^\pm-1$ inverse maps $f_-^{-1}$ or $f_+^{-1}$. The choice and the order of these inverse maps in the composition depends on the combinatorics of the Lorenz map $f$. 

\begin{remark}
  In the remaining part of the paper we identify $f_\pm\circ (p\cR^n f_\pm)^{-1}$ with this composition.
\end{remark}

We note that this composition is defined and homeomorphic on some maximal interval that contains $p\cR^n f_\pm(C_{n\pm})$. The goal of this subsection is to study the properties of this maximal interval.

Consider the finite orbit 
of the interval $C_{n-}$ under the dynamics of $p\cR^{n-1}f$ before its first return to $C_n$. Let $S_{n-}\subset (0,1)$ be the interval from this orbit 
such that $S_{n-}$ lies to the right from the critical point $c$ and is closest to $c$. Similarly, let $S_{n+}\subset (0,1)$ be the interval from the orbit 
of $C_{n+}$ under the dynamics of $p\cR^{n-1}f$ before its first return to $C_n$, such that $S_{n+}$ lies to the left from the critical point $c$ and is closest to it.

\begin{definition}\label{Q_n_def}
We define $Q_{n\pm}$ to be the minimal open interval, containing the intervals $S_{n\pm}$ and $p\cR^nf_\pm(L_{n\pm})$.
\end{definition}

\begin{proposition} 
	Assume, $f\in\mathfrak L$ is an $n$-times renormalizable Lorenz map, for some $n\in\bbN$. 
	Then the map $f_\pm\circ (p\cR^n f_\pm)^{-1}$ is a homeomorphism of $Q_{n\pm}$ onto its image.
\end{proposition}
\begin{proof}
We will give a proof in the case of the interval $Q_{n-}$. The proof for the interval $Q_{n+}$ is analogous.

Consider the orbit of the interval $C_{n-}$ under the dynamics of $f$. Let $m\in\bbN$ be such that $f^m(C_{n-})$ is the first return of this orbit to $C_n$. 
Let $X_2,X_3,\dots,X_m\subset (0,1)$ be the family of open intervals, such that $X_2$ is the minimal interval that contains $f(C_{n-})$ and $f^2(C_{n-})$, and for $k=2,\dots,m-1$,
$$
X_{k+1}=\begin{cases}
f(X_k),& \text{ if } c\not\in X_k,\\
f(\{x\in X_k\mid x<c\}),& \text{ if } c\in X_k.
\end{cases}
$$
In both cases $X_{k+1}$ contains the interval $f^{k+1}(C_{n-})$ and at least one other interval $f^{l}(C_{n-})$, for some $l<k+1$, which lies to the right from $f^{k+1}(C_{n-})$. (The proof is by induction: in the first case of the above formula, $f^{l}(C_{n-})$ is the image of the corresponding interval from $X_k$ under the map $f$, and in the second case $f^{l}(C_{n-}) = f(C_{n-})$.) 

Consider an interval $I=X_m\cup p\cR^nf_-(L_{n-})$. It follows from construction of the interval $X_m$ that the map $f_-\circ(p\cR^nf_-)^{-1}$ is a homeomorphism on $I$, and since $Q_{n-}\subset I$, this completes the proof of the proposition.
\end{proof}

\begin{lemma}\label{L_n_in_P_n_lemma}
Assume, $f\in\mathfrak L$ is an $(n+1)$-times renormalizable Lorenz map, for some $n\in\bbN$. Then
\begin{enumerate}[(i)]
\item \label{L_n_in_P_n_item1} the following inclusions hold:
$$
L_{n-}\cup C_{n+}\Subset Q_{n-}\qquad\text{ and}\qquad L_{n+}\cup C_{n-}\Subset Q_{n+};
$$
\item \label{L_n_in_P_n_item1a} each of the two connected components of the sets
$$
Q_{n\pm}\setminus (L_{n\pm} \cup C_{n\mp})
$$
contain either an interval from the finite orbit of $C_{(n+1)\mp}$ or from the finite orbit of $C_{n\pm}$ under the map $p\cR^{n-1}f$ before their first returns to $C_n$.
\end{enumerate}
\end{lemma}
\begin{proof}
	Parts~\ref{L_n_in_P_n_item1} and~\ref{L_n_in_P_n_item1a}  follow immediately from the observation that one component of $Q_{n\pm}\setminus (L_{n\pm} \cup C_{n\mp})$ contains the interval $S_{n\pm}$ and another component is the interval $B_{n\pm}$ from Lemma~\ref{L_n_combinatorial_lemma}, hence contains an interval from the finite orbit of $C_{(n+1)\mp}$ under the map $p\cR^{n-1}f$ before its first return to $C_n$.
\end{proof}

For every integer $n\in\bbN$ and an $n$-times renormalizable Lorenz map $f$, let $\mathcal O_{n\pm}$ be the finite orbit of the interval $f(\overline{C_{n\pm}})$ until its first return to $C_n$ under the dynamics of $f$. That is,
$$
\mathcal O_{n\pm} = \{f(\overline{C_{n\pm}}), f^2(\overline{C_{n\pm}}),\dots, f^{m_n^\pm}(\overline{C_{n\pm}}) \}.
$$
The elements of $\mathcal O_{n\pm}$ are closed intervals that have pairwise disjoint interiors. Similarly, let $\mathcal Q_{n\pm}$ be the finite orbit of the interval $f_\pm\circ(p\cR^nf_\pm)^{-1}(Q_{n\pm})$ under the dynamics of $f$ until it is mapped onto $Q_{n\pm}$. That is,
\begin{multline}\label{Q_n_orbit}
\mathcal Q_{n\pm} = \{f_\pm\circ(p\cR^nf_\pm)^{-1}(Q_{n\pm}),\, f\circ f_\pm\circ(p\cR^nf_\pm)^{-1}(Q_{n\pm}),\\
f^2\circ f_\pm\circ(p\cR^nf_\pm)^{-1}(Q_{n\pm}),\dots, Q_{n\pm}\}.
\end{multline}

By construction, each interval $f^k\circ f_\pm\circ(p\cR^nf_\pm)^{-1}(Q_{n\pm})$ from $\mathcal Q_{n\pm}$ ($k=0,\dots,m_n^\pm$) contains the interval $f^{k+1}(\overline{C_{n\pm}})$ in its interior. We will say that the latter is the \textit{core subinterval} of the former one.

\begin{lemma} \label{Q_n_orbit_lemma}
Assume, $f\in\mathfrak L$ is an $n$-times renormalizable Lorenz map, for some $n\in\bbN$. Then
\begin{enumerate}[(i)]

\item \label{L_n_in_P_n_item1c} every interval from the finite orbit $\mathcal Q_{n\pm}$ does not contain any other intervals from $\mathcal O_{n\pm}$ in its interior except for its core subinterval.

\item \label{L_n_in_P_n_item1b}
furthermore, every point of the interval $(0,1)$ belongs to no more than three intervals from the finite orbit~$\mathcal Q_{n\pm}$.
\end{enumerate}

\end{lemma}
\begin{proof}
Assume that some interval from $\mathcal Q_{n\pm}$ contains two different intervals from $\mathcal O_{n\pm}$ in its interior. Then so does the interval $Q_{n\pm}$, since $f$ is a homeomorphism on all intervals of~$\mathcal Q_{n\pm}$, except $Q_{n\pm}$. One of the two intervals of $\mathcal O_{n\pm}$ contained in the interior of $Q_{n\pm}$ is its core subinterval. Then, according to the construction of the interval $Q_{n\pm}$ (Definition~\ref{Q_n_def}), the other one is contained in $p\cR^nf_\pm(L_{n\pm})\setminus C_n$, which is not possible. This completes the proof of part~\ref{L_n_in_P_n_item1c}.

Part~\ref{L_n_in_P_n_item1b} follows immediately from part~\ref{L_n_in_P_n_item1c}, since according to part~\ref{L_n_in_P_n_item1c}, every interval from $\mathcal Q_{n\pm}$ has common points with no more than three intervals of $\mathcal O_{n\pm}$, none of them share the same core subinterval and all intervals of $\mathcal O_{n\pm}$ are pairwise disjoint.
\end{proof}

\subsection{Sizes of dynamically important intervals}
In this subsection we combine the combinatorial properties of of Lorenz maps established in the previous subsection, and the Real Koebe Distortion Principle in order to get control on the sizes of dynamically important intervals in the presence of real bounds.

\begin{lemma}\label{Interval_not_too_short_lemma}
	Given a pair of positive numbers $\delta,\Delta>0$ and a finite set $\Theta\subset\mathfrak P$, there exist positive real constants $\beta_1=\beta_1(\delta,\Delta,\Theta)$ and $\beta_2=\beta_2(\delta,\Delta,\Theta)$, such that $0<\beta_1<\beta_2<1$ and for any twice renormalizable Lorenz map $f\in\mathcal S_{\Theta}^2$ with real $(\delta,\Delta)$-bounds of level $0$, the following holds: if $I$ is any interval from the orbits of $C_{1+}$ or $C_{1-}$ before their first return to $C_1$ or any interval from the orbits of $C_{2+}$ or $C_{2-}$ before their first return to $C_2$, and $J$ is one of the intervals $C_+$ or $C_-$, such that $I\subset J$, then
	$$
	\beta_1<|I|/|J|<\beta_2.
	$$
\end{lemma}
\begin{proof}
	Since $\Theta$ is a finite set, the prerenormalization $p\cR^2 f_\pm=f^{k_\pm}$ is a finite composition, where $k_\pm\le B$, for some constant $B=B(\Theta)$. Furthermore, we have
	$$
	|f^{k_\pm}(C_{2\pm})| \ge |C_{2\pm}|.
	$$
	On the other hand, it follows from Lemma~\ref{derivative_bounds_lemma} that there exist real constants $K_1=K_1(\delta,\Delta)>0$ and $K_2=K_2(\delta,\Delta)>0$, such that for any $f\in\mathbf L$ with real $(\delta,\Delta)$-bounds of level $0$ and any $x\in[0,1]$, we have 
	$$
	f'(x)<K_1\qquad\text{and}\qquad |f_\pm(x)-f_\pm(c_f)|\le K_2|x-c_f|^{\alpha}.
	$$
	Now we have $|f(C_{2\pm})|\le K_2 |C_{2\pm}|^{\alpha}$, and 
	$$
	|f^{k_\pm}(C_{2\pm})| <K_1^{B-1} K_2 |C_{2\pm}|^{\alpha}.
	$$
	Since $\alpha>1$, if $C_{2\pm}$ is too short, then the right-hand side of the last inequality is smaller than $|C_{2\pm}|$, which is a contradiction. This implies that there is a lower bound on the lengths of the intervals $|C_{2\pm}|$, hence also a lower bound on the lengths of the intervals $|f^{k_\pm}(C_{2\pm})|$. The latter together with the upper bound on $f'$ implies existence of a lower bound on the lengths of all intervals from the orbits of $C_{2\pm}$ before their return to $C_2$. Since $|J|\ge\delta$, we conclude that there exists $\beta_1>0$, such that 
	$$
	|I|/|J|>\beta_1.
	$$
	Finally, we may choose $\beta_2>1-\beta_1$.
	
	The proof for the orbits of the intervals $C_{1\pm}$ is analogous.
\end{proof}

Now we recall the Macroscopic Koebe Principle (c.f. Section~IV.3 of~\cite{deMelo_vanStrien}). We state it for Lorenz maps from $\mathfrak L$, however it holds for a much wider class of maps. 

If $I\subset J$ are two intervals and $\tau>0$ is a real number, we say that $J$ contains a $\tau$-scaled neighborhood of $I$ if each of the two component of $J\setminus I$ has at least length $\tau|J|$.

\begin{theorem}[\textbf{Macroscopic Koebe Principle}]\label{Macro_Koebe_theorem}
Given $f\in\mathfrak L$, there exists a strictly positive function $B_0\colon\bbR^+\to\bbR^+$ such that for any pair of intervals $J\subset T$, any $m\ge 0$ and any $0<\tau <1$, if the following conditions are satisfied:
\begin{enumerate}
	\item $f^m|T$ is a diffeomorphism;
	\item $f^m(T)$ contains a $\tau$-scaled neighborhood of $f^m(J)$;
	\item $\sum_{i=0}^{m-1} |f^i(T)|\le 3$;
\end{enumerate}
then $T$ contains a $B_0(\tau)$-scaled neighborhood of $J$.
\end{theorem}
\begin{remark}\label{Koebe_dependence_remark}
It follows from the proof of the Macroscopic Koebe Principle (c.f. Section~IV.3 of~\cite{deMelo_vanStrien}) that the function $B_0$ depends on the constants $K_1$ and $K_2$ from~(\ref{eta_f_bounds}). Hence, according to Lemma~\ref{derivative_bounds_lemma}, for any pair of real numbers $\delta,\Delta>0$, the function $B_0$ can be chosen uniformly over the class of all at least once renormalizable $f\in\mathfrak L$ having real $(\delta,\Delta)$-bounds of level $0$.
\end{remark}

We use the Macroscopic Koebe Principle together with the combinatorial analysis of Lorenz maps to prove the following Lemma:

\begin{lemma}\label{O_n_Q_n_decrease_lemma}
For any pair of positive numbers $\delta,\Delta>0$ and a finite set $\Theta\subset\mathfrak P$, there exists a positive real constant $\beta_3=\beta_3(\delta,\Delta,\Theta)$ such that $0<\beta_3<1$, and for any $n\in\bbN$ and any $f\in\mathcal S_\Theta^{n+1}$ with real $(\delta,\Delta)$-bounds of level $n$, every interval from the orbits $\mathcal O_{n\pm}$ and $\mathcal Q_{n\pm}$ has length less than $\beta_3^n$.
\end{lemma}
\begin{proof}
First, we observe that according to part~\ref{L_n_in_P_n_item1a} of Lemma~\ref{L_n_in_P_n_lemma} and Lemma~\ref{Interval_not_too_short_lemma}, there exists $\tau=\tau(\delta,\Delta,\Theta)$ such that for any $k=1,\dots,n$, the interval $Q_{k\pm}$ from the orbit $\mathcal Q_{k\pm}$ contains a $\tau$-scaled neighborhood of $C_k$, hence also a $\tau$-scaled neighborhood of the corresponding core subinterval $f^{m_k^\pm}(\overline{C_{k\pm}})$. Next, we note that according to part~\ref{L_n_in_P_n_item1b} of Lemma~\ref{Q_n_orbit_lemma}, the sum of the lengths of all intervals from the orbit $\mathcal Q_{k\pm}$ is not greater than $3$, hence it follows from the Macroscopic Koebe Principle and Remark~\ref{Koebe_dependence_remark} that there exists a constant $B_0=B_0(\delta,\Delta,\tau)$, such that every interval from the orbit $\mathcal Q_{k\pm}$ contains a $B_0$-scaled neighborhood of its core subinterval.

Next we will show that for every $k=1,\dots,n-1$, the $B_0$-scaled neighborhood of every interval from the orbit $\mathcal Q_{(k+1)\pm}$ is contained in some interval of the orbits $\mathcal Q_{k+}$ or $\mathcal Q_{k-}$.

Consider the interval $A_{(k+1)\pm}=p\cR^kf_\pm(p\cR^{k+1}f_\pm)^{-1}(Q_{(k+1)\pm})$ from the orbit $\mathcal Q_{(k+1)\pm}$. Since the interval $A_{(k+1)\pm}$ is eventually mapped homeomorphically onto $Q_{(k+1)\pm}$ by the dynamics of $p\cR^kf$, it follows that $A_{(k+1)\pm}\subset C_k$ and $c_f\not\in A_{(k+1)\pm}$. Thus, $f(A_{(k+1)\pm})$ and all its further iterates under the dynamics of $f$ until the return to $Q_{(k+1)\pm}$ are contained in the core subintervals of some intervals of the orbits $\mathcal Q_{k+}$ and $\mathcal Q_{k-}$. Hence, the above statement will hold for these intervals.

Finally, we observe that $A_{(k+1)\pm}\subset C_k\subset Q_{k\pm}$, so $Q_{k\pm}$ contains a $\tau$-scaled neighborhood of $A_{(k+1)\pm}$, hence, by the Macroscopic Koebe Principle, the $B_0$-scaled neighborhoods of all intervals of the orbit $\mathcal Q_{(k+1)\pm}$ before and including $A_{(k+1)\pm}$ are contained in the corresponding intervals of the orbit $\mathcal Q_{k\pm}$.

We complete the proof by choosing $\beta_3=B_0$.
\end{proof}

Combining the results of this section and Lemma~\ref{definite_neighb_lemma}, we obtain the following important result:
\begin{lemma}\label{pRnf_distorted_root_lemma}
	For any pair of positive numbers $\delta,\Delta>0$, a finite set $\Theta\subset\mathfrak P$ and a real number $t\in\bbR$ such that $0<t<1$, there exist a positive number $\mu_0=\mu_0(\delta,\Delta, \Theta, t)$, such that for every real number $r>0$, there exists $n_0=n_0(r,t,\delta,\Delta,\Theta)\in\bbN$ with the property that for all $n\ge n_0$ and $f\in\mathfrak L_r\cap\mathcal S_\Theta^{n+1}$ with real $(\delta,\Delta)$-bounds of level $n$, the map $(p\cR^n f_\pm)^{-1}$ is a distorted root of degree $\alpha$ on $D_t(L_{n\pm}\cup C_{n\mp})$ with modulus $\mu_0$, precomposed and postcomposed with some affine maps.	
\end{lemma}
\begin{proof}
It follows from part~\ref{L_n_in_P_n_item1b} of Lemma~\ref{Q_n_orbit_lemma} that the total length of all intervals from the finite orbit $\mathcal Q_{n\pm}$
is not greater than~$3$. At the same time, Lemma~\ref{O_n_Q_n_decrease_lemma} implies that the length of the longest interval from the finite orbit $\mathcal Q_{n\pm}$ converges to zero uniformly in $f$, as $n\to\infty$. These observations together with Lemma~\ref{definite_neighb_lemma} imply that there exists $n_0=n_0(r,t,\delta,\Delta,\Theta)\in\bbN$, such that for all $n\ge n_0$, the map $f_\pm\circ (p\cR^n f_\pm)^{-1}$ is defined on $D_t(Q_{n\pm})$ and
$$
f_\pm\circ (p\cR^n f_\pm)^{-1}(D_t(Q_{n\pm}))\subset D_{t/2}(f_\pm\circ (p\cR^n f_\pm)^{-1}(Q_{n\pm})).
$$

Due to condition~\ref{L_r_def_item2} of Definition~\ref{L_r_def}, and the fact that the length of the intervals $f_\pm\circ (p\cR^n f_\pm)^{-1}(Q_{n\pm})$ converges to zero uniformly in $f$, we may assume without loss of generality that $n_0$ is large enough, so that if $n\ge n_0$, then the inverse map $f_\pm^{-1}$ is a root of degree~$\alpha$ on $D_{t/2}(f_\pm\circ (p\cR^n f_\pm)^{-1}(Q_{n\pm}))$ precomposed and postcomposed with some conformal maps. Together with the above inclusion this implies that $(p\cR^n f_\pm)^{-1}$ is a root of degree $\alpha$ on $D_t(Q_{n\pm})$ up to a precomposition and a postcomposition with conformal maps.

Finally, it follows from part~\ref{L_n_in_P_n_item1a} of Lemma~\ref{L_n_in_P_n_lemma} and Lemma~\ref{Interval_not_too_short_lemma} that there exists a positive real number $\mu_0=\mu_0((\delta,\Delta, \Theta, t)$, such that $\mod(D_t(L_{n\pm}\cup C_{n\mp}),D_t(Q_{n\pm}))\ge \mu_0$, for all Lorenz maps $f$ satisfying the conditions of Lemma~\ref{pRnf_distorted_root_lemma}. Hence, it follows that the map $(p\cR^n f_\pm)^{-1}$ is a distorted root of degree $\alpha$ on $D_t(L_{n\pm}\cup C_{n\mp})$ with modulus $\mu_0$, precomposed and postcomposed with some affine maps.
\end{proof}

\section{Proofs of main results}\label{Main_proofs_sec}

In this section we give proofs of our main results by combining the complex analytic tools from Section~\ref{Complex_neib_sec} with combinatorial and metric properties of Lorenz maps on the real line, established in Section~\ref{dyn_intervals_sec}.

For a positive integer $n\in\bbN$ and a Lorenz map $f\in\mathcal S_{\mathfrak P}^n$, let the interval $L_n$ be the union $L_n=L_{n+}\cup L_{n-}\cup\{c\}$. We define the set $D_n$ as the hyperbolic neighborhood $D_n=D_\sigma(L_n)\subset\bbC$. 

A key step in the proof of Theorem~\ref{main_theorem} is the following lemma. Its proof will be given later.

\begin{lemma}[\textbf{Main Lemma}]\label{main_main_lemma}
	For any pair of positive numbers $\delta,\Delta>0$ and a finite set $\Theta\subset\mathfrak P$, there exists a constant $B_1>0$, such that for each real number $r>0$ and a positive integer $m\in\bbN$, there exists $n_0=n_0(r,m,\delta,\Delta,\Theta)\in\bbN$ with the property that for all $n\ge n_0$ and $f\in\mathfrak L_r\cap\mathcal S_\Theta^{n+1}$ with real $(\delta,\Delta)$-bounds of level $n$, the inverse maps $(p\cR^n f_\pm)^{-1}$ have well defined univalent analytic extensions to $D_{n-m}\setminus\bbR$, and 
	for all $z\in D_{n-m}\setminus\bbR$ we have
	\begin{equation}\label{main_inequality}
	|(p\cR^n f_\pm)^{-1}(z) - c|\le B_1\frac{|L_{n-m}|^{1/\alpha}}{|C_{n\pm}|^{(1-\alpha)/\alpha}}.
	\end{equation}
\end{lemma}

\subsection{Flowers}

We fix a positive real number $\sigma<\cot\left(\frac{\pi}{2\alpha}\right)$ that remains unchanged until the end of the paper.

\begin{definition}
	A set $F\subset\bbC$ is a flower of an interval $(a,b)\subset\bbR$, if there exist real numbers $d,e,t$, such that $a<d<e<b$ and $t>0$, and
	\begin{equation}\label{Flower_form_eq}
	F=D_\sigma((a,d))\bigcup D_t((d,e))\bigcup D_\sigma((e,b)).
	\end{equation}
	For a real number $K>0$ we say that a flower $F$ is $K$-bounded, if
	$$
	\frac{d-a}{b-a}\ge K \qquad\text{and}\qquad \frac{b-e}{b-a}\ge K.
	$$
	The real number $t$ will be called the parameter of the flower $F$.
\end{definition}


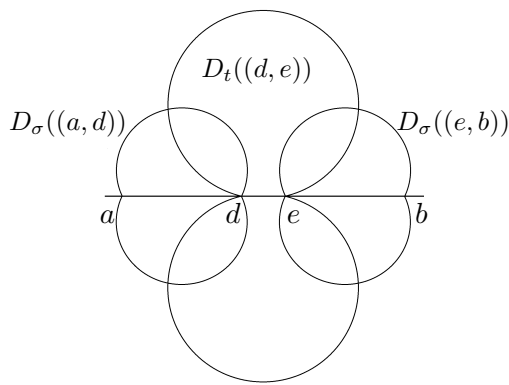
\begin{figure}[ht]
\centering
\includegraphics[height=5cm]{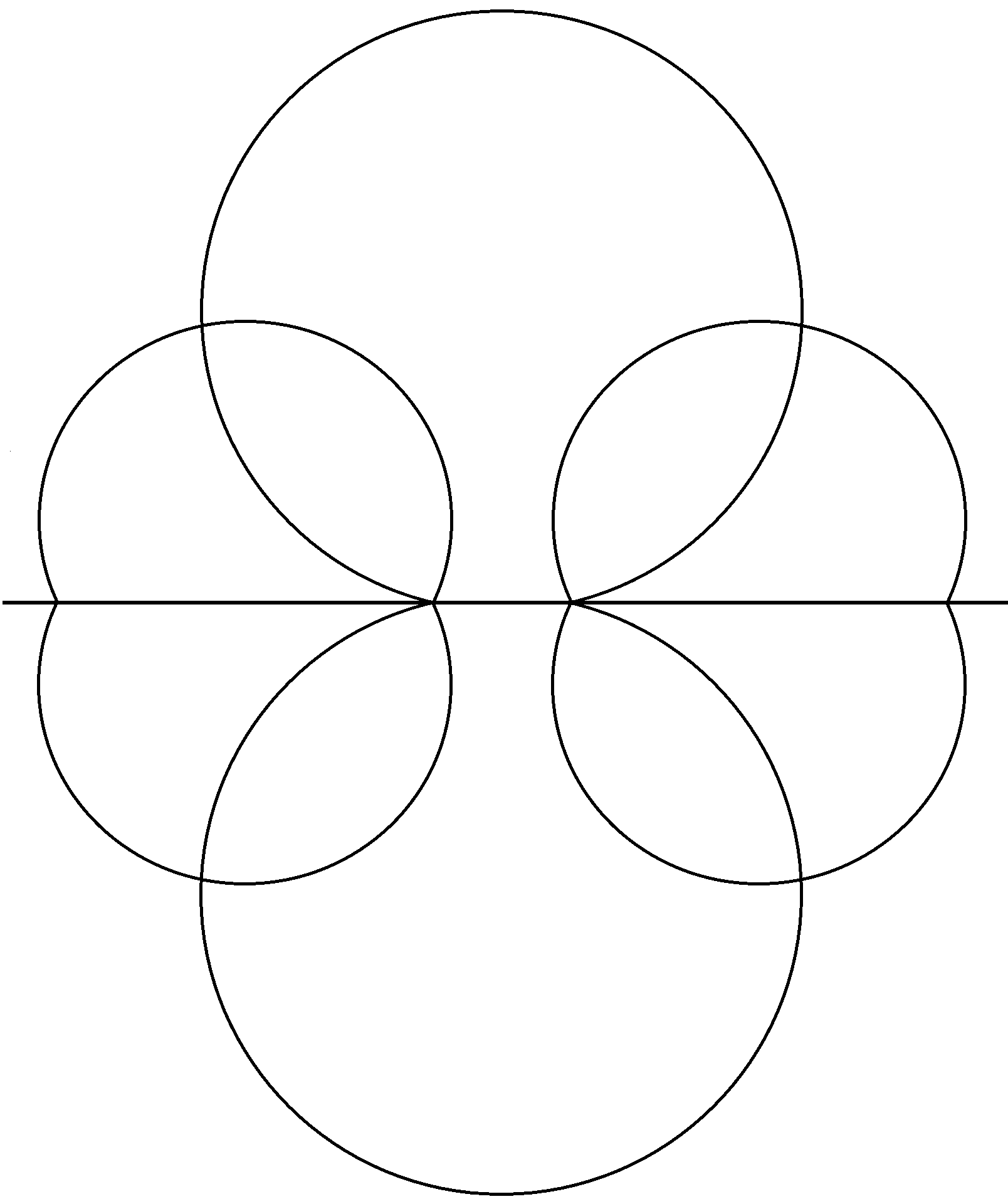} 
\put(-121,61){\small $a$}
\put(-74,61){\small $d$}
\put(-51,61){\small $e$}
\put(-3,61){\small $b$}
\put(-10,96){\footnotesize $D_\sigma((e,b))$}
\put(-155,96){\footnotesize $D_\sigma((a,d))$}
\put(-83,116){\footnotesize $D_t((d,e))$}
\caption{A flower.}
\end{figure}

\begin{proposition}\label{diam_bound_prop}
	For any $K, t>0$, there exists a real number $b=b(K,t)>0$, such that if $F$ is a $K$-bounded flower of an interval $I$ with parameter $t$, then for any subinterval $J\subset I$, there exists $\hat t>0$, such that $F\setminus D_\sigma(I)\subset D_{\hat t}(J)$, and $\diam[D_{\hat t}(J)]<b|I|$.
\end{proposition}
\begin{proof}
	For any point $z\in F\setminus D_\sigma(I)$, consider the triangle with side $J$ and the opposite vertex at $z$. The lengths of other sides of this triangle are smaller than $b_1|I|$, for some constant $b_1=b_1(t)>0$, and the angles $\alpha,\beta$ at the opposite vertices satisfy the inequality $\epsilon<\alpha,\beta<\pi-\epsilon$, for some $\epsilon=\epsilon(K,t)>0$. Hence, the diameter of the circumscribed circle of the considered triangle is less than $b_1|I|/\sin\epsilon$. Now the proposition follows.
\end{proof}

Let $f$ be a Lorenz map with the critical point $c\in\bbR$. For real numbers $K_1,K_2>0$, we say that a flower $F$ of an interval $(a,b)$, defined as in~(\ref{Flower_form_eq}), is $(K_1,K_2)$-bounded, if either $a=c$ and the inequalities 
$$
\frac{d-a}{b-a}\ge K_1, \qquad\text{}\qquad \frac{b-e}{b-a}\ge K_2
$$
hold, or if $b=c$ and the inequalities 
$$
\frac{d-a}{b-a}\ge K_2, \qquad\text{}\qquad \frac{b-e}{b-a}\ge K_1
$$
hold. We will say that $K_1$ is the \textit{critical bound} and $K_2$ is the \textit{non-critical bound}.

We note that if $f\in\mathcal S_{\mathfrak P}^{n}$, then the corresponding intervals $L_{n+}$ and $L_{n-}$ are defined and have the critical point $c$ as one of their boundary points. Hence, one can consider $(K_1,K_2)$-bounded flowers of these intervals.

For $n\in\bbN$ and $f\in\mathcal S_{\mathfrak P}^n$, let $T^nf_-\colon C_{n-}\to C_{n-1}$ and $T^nf_+\colon C_{n+}\to C_{n-1}$ be the maps, such that
\begin{equation}\label{T_n_def_eq}
p\cR^{n} f_\pm= p\cR^{n-1}f\circ T^{n} f_\pm.
\end{equation}
In particular, $T^n f_\pm$ is a finite composition of prerenormalizations $p\cR^{n-1}f$.

For $n\in\bbN$ and $f\in\mathcal S_{\mathfrak P}^{n+1}$, let the maps $f_{n\pm}$ be defined by 
$$
f_{n\pm}=p\cR^{n+1} f_\pm \circ (T^{n+1} f_\pm)^{-1}.
$$
Then each $f_{n\pm}$ is either $p\cR^{n}f_-$ or $p\cR^{n}f_+$, depending on the combinatorics of the map $f$, and 
$$
p\cR^{n+1} f_\pm= f_{n\pm}\circ T^{n+1} f_\pm.
$$
Let the intervals $L_{n\pm}^f$ be defined by
\begin{equation*}
L_{n\pm}^f=\begin{cases}
       L_{n-}, & \text{if } f_{n\pm}= p\cR^{n}f_- \\
       L_{n+}, & \text{if } f_{n\pm}= p\cR^{n}f_+. 
      \end{cases}
\end{equation*}

Recall that for a positive integer $n\in\bbN$ and a Lorenz map $f\in\mathcal S_{\mathfrak P}^n$, we have $L_n:=L_{n+}\cup L_{n-}\cup\{c\}$ and $D_n:=D_\sigma(L_n)$. The following lemma provides the induction step in the proof of Lemma~\ref{main_main_lemma}.

\begin{lemma}\label{induction_lemma}
For any pair of positive numbers $\delta,\Delta>0$ and a finite set $\Theta\subset\mathfrak P$, there exist 
real constants $\tl t,K_1,K_2>0$, such that 
for every real number $r>0$, there exists $n_1=n_1(r,\delta,\Delta,\Theta)\in\bbN$ with the property that for all $n\ge n_1$ and $f\in\mathfrak L_r\cap\mathcal S_\Theta^{n+1}$ with real $(\delta,\Delta)$-bounds of level $n$, 
the following holds:

(i) The inverse maps $(p\cR^n f_\pm)^{-1}$ 
have well defined univalent analytic extensions to $D_{n}\setminus\bbR$, and the preimage 
$(p\cR^n f_\pm)^{-1}(D_{n}\setminus\bbR)$ is contained in a $(K_1,K_2)$-bounded flower of the interval $L_{n\pm}$ with parameter~$\tl t$. 

(ii) Let $F$ be any $(K_1,K_2)$-bounded flower of the interval $L_{n\pm}^f$ with parameter~$\tl t$. 
Then the inverse map $(T^{n+1} f_\pm)^{-1}$ has a well defined univalent analytic extension to $D_\sigma(L_{n\pm}^f)\cap F\setminus \bbR$, and the preimage $(T^{n+1} f_\pm)^{-1}(D_\sigma(L_{n\pm}^f)\cap F\setminus \bbR)$ is contained in a $(K_1,K_2)$-bounded flower of the interval $L_{n+1\pm}$ with parameter~$\tl t$. 
\end{lemma}
\begin{proof}
In order to prove part (i), we notice that due to finiteness of the set $\Theta$, the map $(p\cR^nf_\pm)^{-1}$ is a composition of no more than $k$ inverse prerenormalizations $(p\cR^{n-1}f_-)^{-1}$ and $(p\cR^{n-1}f_+)^{-1}$, where $k$ depends only on $\Theta$. According to Lemma~\ref{pRnf_distorted_root_lemma}, these inverse prerenormalizations are distorted roots of degree~$\alpha$ on hyperbolic neighborhoods $D_t(C_{n-1})$ of $C_{n-1}$, for all sufficiently large $n$. Then part~(i) of Lemma~\ref{induction_lemma} will follow from applying Lemma~\ref{root_lemma_bounded_distortion} $k$ times.

The proof of part~(ii) of Lemma~\ref{induction_lemma} is based on the idea that according to~(\ref{T_n_def_eq}), the map $(T^{n+1}f_\pm)^{-1}$ is a finite composition of inverse branches of prerenormalizations $(p\cR^{n}f_-)^{-1}$ and $(p\cR^{n}f_+)^{-1}$. Since the set of combinatorics $\Theta$ is finite, the number of these maps in the composition is less than a constant $B=B(\Theta)>0$, hence $(T^{n+1}f_\pm)^{-1}$ is a finite composition of at most $kB$ inverse branches of prerenormalizations $(p\cR^{n-1}f_-)^{-1}$ and $(p\cR^{n-1}f_+)^{-1}$. 
The rest of the proof is left to the reader as it is analogous to the proof of part~(i) of Lemma~\ref{induction_lemma}. (The constant $\tilde t$ and the non-critical bound $K_2$ might have to be further decreased.)
\end{proof}

\begin{proof}[Proof of Lemma~\ref{main_main_lemma}]
We will give a proof for the case of positive branches $(p\cR^n f_+)^{-1}$. The case of negative branches $(p\cR^n f_-)^{-1}$ is analogous.

We fix $r$, $m$, $\delta$, $\Delta$, $\Theta$. As a first step, we will prove that there exists a real number $B_3=B_3(\delta, \Delta,\Theta)>0$, such that for any $z\in D_{n-m}\setminus\bbR$, the point $f_+\circ (p\cR^n f_+)^{-1} (z)$ is defined and satisfies the inequality
\begin{equation}\label{Lnm_inequality}
|f_+\circ (p\cR^n f_+)^{-1} (z) -f_+(c)| \le B_3 \frac{|L_{n-m}||f(C_{n+})|}{|C_{n+}|}.
\end{equation}

It follows from~(\ref{T_n_def_eq}) that for every $n\ge m$ and every $f\in\mathcal S_\Theta^n$, the prerenormalization $p\cR^nf_+$ can be represented in a unique way as a composition
$$
p\cR^n f_+= f_{n,m}\circ h_{n,m-1}\circ h_{n,m-2}\circ\dots\circ h_{n,0},
$$
where $f_{n,m}$ is either $p\cR^{n-m}f_+$ or $p\cR^{n-m}f_-$ and for each $k=0,\dots,m-1$, the map $h_{n,k}$ is either $T^{n-k}f_+$ or $T^{n-k}f_-$.
Furthermore,~(\ref{T_n_def_eq}) implies that for every $k=0,\dots, m-1$, we have
$$
f_{n,k}=f_{n,m}\circ h_{n,m-1}\circ\dots\circ h_{n,k},
$$
and $f_{n,k}$ is either $p\cR^{n-k}f_+$ or $p\cR^{n-k}f_-$.
For $k=0,\dots, m$, let the intervals $L_{n,k}$ be defined by  
\begin{equation*}
L_{n,k}=\begin{cases}
       L_{(n-k)-}, & \text{if } f_{n,k}= p\cR^{n-k}f_- \\
       L_{(n-k)+}, & \text{if } f_{n,k}= p\cR^{n-k}f_+. 
      \end{cases}
\end{equation*}

\begin{proposition}\label{mainlemma_prop}
Let $r$, $m$, $\delta$, $\Delta$ and $\Theta$ be the same as in Lemma~\ref{main_main_lemma}. Then for any pair of real numbers $K,t>0$, there exists a positive integer $n_2=n_2(r,\delta,\Delta,\Theta, m, K, t)$ and a positive real number $B=B(\delta,\Delta,\Theta,K,t)$, such that for any $n\ge n_2$ and $f\in\mathfrak L_r\cap\mathcal S_\Theta^{n+1}$ with real $(\delta,\Delta)$-bounds of level $n$, the following holds: for any $k=0,\dots,m$, if $F$ is a $K$-bounded flower on $L_{n,k}$ with parameter $t$, and 
$z\in\bbC$ is such that $f_{n,k}^{-1}(z)$ is defined and either $k\ge 1$ and $f_{n,k}^{-1}(z)\in F\setminus D_\sigma(L_{n,k})$ or $k=0$ and $f_{n,k}^{-1}(z)\in F$, then $f_+\circ (p\cR^n f_+)^{-1} (z)$ is defined and
$$
|f_+\circ (p\cR^n f_+)^{-1} (z) -f_+(c)| \le B \frac{|L_{n-m}||f(C_{n+})|}{|C_{n+}|}.
$$
\end{proposition}
\begin{proof}
First, consider the case $k\ge 1$. Then $f_{n,k}^{-1}(z)\in F\setminus D_\sigma(L_{n,k})$ and according to Proposition~\ref{diam_bound_prop}, there exists $0<\hat t<1$, such that $f_{n,k}^{-1}(z)$ is contained in the hyperbolic neighborhood 
$$
D_{\hat t}(f_{n,k}^{-1}(p\cR^nf(C_{n+})))
$$
whose diameter is less than $B_4|L_{n,k}|$, for some constant $B_4=B_4(K,t)>0$. 
Since $k\le m$, Lemma~\ref{Interval_not_too_short_lemma} implies that 
there exists a real constant $\mu=\mu(\delta,\Delta,\Theta)>0$, such that 
$$
|p\cR^nf(C_{n+})|>|L_{n,k}|\mu^{k}\ge |L_{n,k}|\mu^{m}.
$$
Applying Lemma~\ref{derivative_bounds_lemma} to the map $f_{n,k}$, we obtain that
$$
|f_{n,k}^{-1}(p\cR^nf(C_{n+}))|> B_5|p\cR^nf(C_{n+})| > B_5 |L_{n,k}|\mu^{m},
$$
for some constant $B_5=B_5(\delta,\Delta,\Theta)>0$.
The latter implies that the parameter $\hat t>0$ is bounded away from zero uniformly with respect to the choice of $f$ and $z$. At the same time, Lemma~\ref{O_n_Q_n_decrease_lemma} implies that the lengths of all intervals from the orbits of $C_{n+}$ and $C_{n-}$ converge to zero as $n\to\infty$, so according to Lemma~\ref{definite_neighb_lemma},
there exists a positive integer $n_2=n_2(r,\delta,\Delta,\Theta, m, K, t)$ 
such that if $n\ge n_2$, then 
$$
f_+\circ (p\cR^n f_+)^{-1} (z)\in D_{\hat t/2}(f(C_{n+})),
$$
which implies that
$$
\frac{|f_+\circ (p\cR^n f_+)^{-1} (z) -f_+(c)|}{|f(C_{n+})|}\le
\frac{\diam [D_{\hat t/2}(f_{n,k}^{-1}(p\cR^nf(C_{n+})))]}{|f_{n,k}^{-1}(p\cR^nf(C_{n+}))|}.
$$
We note that 
$$
\diam [D_{\hat t/2}(f_{n,k}^{-1}(p\cR^nf(C_{n+})))]\le 4 \diam [D_{\hat t}(f_{n,k}^{-1}(p\cR^nf(C_{n+})))]\le 4B_4|L_{n,k}|,
$$
and since
$$
|f_{n,k}^{-1}(p\cR^nf(C_{n+}))|> B_5|p\cR^nf(C_{n+})| > B_5 |C_{n+}|,
$$
there exists a real number $B=B(\delta,\Delta,\Theta,K,t)>0$, such that 
\begin{equation*}
\frac{|f_+\circ (p\cR^n f_+)^{-1} (z) -f_+(c)|}{|f(C_{n+})|}\le
\frac{B|L_{n,k}|}{|C_{n+}|}\le \frac{B|L_{n-m}|}{|C_{n+}|},
\end{equation*}
which completes the proof of Proposition~\ref{mainlemma_prop} in case of $k\ge 1$.

If $k=0$, then $f_{n,k}=p\cR^n f_+$ and $L_{n,k}=L_{n+}$. Since $(p\cR^n f_+)^{-1}(z)$ belongs to a flower $F$ on $L_{n+}$ with parameter $t$, and $|L_{n+}|$ is commensurable with $|C_{n+}|$ (c.f. Lemma~\ref{Interval_not_too_short_lemma}), there exists a constant $R=R(\delta,\Delta,\Theta, t)>0$, such that
$$
|(p\cR^nf_+)^{-1}(z)-c|\le R|C_{n+}|.
$$

Due to condition~\ref{L_r_def_item2} of Definition~\ref{L_r_def}, and the fact that the lengths of the intervals $L_{n+}$ converge to zero uniformly in $f$ as $n\to\infty$ (c.f. Lemma~\ref{O_n_Q_n_decrease_lemma}), we may assume without loss of generality that $n_2$ is large enough, so that if $n\ge n_2$, then the map $f_+$ is the power map $z\mapsto z^\alpha$ on $F$ up to a precomposition and a postcomposition with conformal maps of bounded distortion. Hence, there exists a constant $\tl R>0$, such that 
$$
|f_+\circ(p\cR^nf_+)^{-1}(z)-f_+(c)|\le \tl R|f_+(C_{n+})|\le \tl R \frac{|L_{n-m}||f(C_{n+})|}{|C_{n+}|}.
$$
Without loss of generality we may assume that $B\ge\tl R$, which completes the proof of Proposition~\ref{mainlemma_prop}.
\end{proof}

Now we are ready to prove inequality~(\ref{Lnm_inequality}) for all $z\in D_{n-m}\setminus\bbR$. Let the constants $\tl t, K_1,K_2$ be the same as in Lemma~\ref{induction_lemma}. Define
$$
n_0=\max\{n_1(r,\delta,\Delta,\Theta),n_2(r,\delta,\Delta,\theta,m,\min(K_1,K_2),\tl t)\}+m,
$$
where $n_1$ and $n_2$ are the same as in Lemma~\ref{induction_lemma} and Proposition~\ref{mainlemma_prop} respectively. We also define 
$$
B_3=B(\delta,\Delta,\Theta,\min(K_1,K_2),\tl t),
$$
where $B$ is the same as in Proposition~\ref{mainlemma_prop}. 
We will prove by finite induction that if $n\ge n_0$, then for any $k=0,\dots,m$, the inverse maps $f_{n,k}^{-1}$ are well defined on $D_{n-m}\setminus\bbR$ and for any $z\in D_{n-m}\setminus\bbR$, either~(\ref{Lnm_inequality}) holds, or 
\begin{equation}\label{D_sigma_inclusion_eq}
f_{n,k}^{-1}(z)\in D_\sigma(L_{n,k})\cap F_k,
\end{equation}
where $F_k$ is a $(K_1,K_2)$-bounded flower of the interval $L_{n,k}$ with parameter $\tl t$.

The base of induction, the case $k=m$, is given by first applying part~(i) of Lemma~\ref{induction_lemma} and then Proposition~\ref{mainlemma_prop}. The induction step goes from $k$ to $k-1$ as follows: if the above statement holds for some value of $k=l>0$, then either~(\ref{Lnm_inequality}) holds and the statement holds for all $k=0,\dots,m$, or~(\ref{D_sigma_inclusion_eq}) holds and then, by first applying part~(ii) of Lemma~\ref{induction_lemma} and then Proposition~\ref{mainlemma_prop}, we obtain the above statement for $k=l-1$.

Finally, we observe that according to Proposition~\ref{mainlemma_prop}, if $k=0$, then~(\ref{D_sigma_inclusion_eq}) implies~(\ref{Lnm_inequality}), so we proved~(\ref{Lnm_inequality}) for all $z\in D_{n-m}\setminus\bbR$.

We finish the proof of Lemma~\ref{main_main_lemma} by observing that due to condition~\ref{L_r_def_item2} of Definition~\ref{L_r_def}, and the fact that the lengths of the intervals $L_{n-m}$ converge to zero uniformly in $f$ as $n\to\infty$ (c.f. Lemma~\ref{O_n_Q_n_decrease_lemma} combined with part~\ref{L_n_in_P_n_item1} of Lemma~\ref{L_n_in_P_n_lemma}), we may assume without loss of generality that $n_0$ is large enough, so that if $n\ge n_0$, then the map $f_+^{-1}$ restricted to the disk
$$D\left(f_+(c),\,\, B_3 \frac{|L_{n-m}||f(C_{n+})|}{|C_{n+}|}\right),$$ 
is a root of degree $\alpha$, precomposed and postcomposed with conformal maps of bounded distortion. Hence, together with inequality~(\ref{Lnm_inequality}), this implies that there exists a constant $B_1=B_1(\delta,\Delta,\Theta)>0$, such that 
$$
\frac{|(p\cR^n f_+)^{-1} (z) -c|}{|C_{n+}|} \le B_1 \left(\frac{|L_{n-m}|}{|C_{n+}|}\right)^{1/\alpha}.
$$
This completes the proof of Lemma~\ref{main_main_lemma}.
\end{proof}

\begin{proof}[Proof of Theorem~\ref{main_theorem}]
	It follows from Lemma~\ref{Interval_not_too_short_lemma} that for each pair of real numbers $\delta,\Delta>0$ and a finite set $\Theta\subset\mathbf P$, there exist real numbers $\mu_1>\mu_2>1$, such that for any positive integers $n>m>0$ and $f\in \mathcal S_\Theta^{n+1}$ with real $(\delta,\Delta)$-bounds of level $n$, we have
	\begin{equation}\label{mu1mu2_inequality}
	\mu_2^m<\frac{|L_{n-m}|}{|C_{n\pm}|}<\mu_1^m.
	\end{equation}
	The first part of this inequality implies that for any real number $\rho>0$, there exists a positive integer $m=m(\rho, \delta,\Delta, \Theta)>0$, such that 
	\begin{equation}\label{B_1_rho_inequality}
	B_1\frac{|L_{n-m}|^{1/\alpha}}{|C_{n\pm}|^{(1-\alpha)/\alpha}} <\rho|L_{n-m}|,
	\end{equation}
	where $B_1=B_1(\delta,\Delta,\Theta)$ is the same as in Lemma~\ref{main_main_lemma}.
	
	For a complex number $z_0\in\bbC$ and a positive real number $r>0$, let $D(z_0,r)\subset\bbC$ denote the open disk of radius $r$, centered at $z_0$. 	
	It follows from~(\ref{mu1mu2_inequality}) that the critical point $c$ splits the interval $L_{n-m}$ into two commensurable subintervals, so one can choose $\rho=\rho(\delta,\Delta,\Theta)>0$ so that
	$$
	D(c,\rho|L_{n-m}|)\Subset D_{n-m}
	$$
	and 
	$$
	\mod(D(c,\rho|L_{n-m}|), D_{n-m})>\nu>0,
	$$
	for some fixed constant $\nu>0$.
	
	Fix the domain $\tl D=D_{n-m}$. Then, according to Lemma~\ref{main_main_lemma}, there exists $n_0=n_0(r,m(\rho(\delta,\Delta,\Theta)),\delta,\Delta,\Theta)$, such that if $n\ge n_0$, then the maps $(p\cR^nf_\pm)^{-1}$ are defined on $\tl D\setminus\bbR$. We set
	$$
	\tl U_+=(p\cR^nf_+)^{-1}(\tl D\setminus\bbR)\cup L_{n+}\qquad\text{and}\qquad
	\tl U_-=(p\cR^nf_-)^{-1}(\tl D\setminus\bbR)\cup L_{n-}.
	$$
	Then, it follows from Lemma~\ref{main_main_lemma} and inequality~(\ref{B_1_rho_inequality}) that
	$$
	\mod(\tl U_+\cup\tl U_-,\tl D)>\nu.
	$$
	Furthermore, according to Lemma~\ref{induction_lemma}, the domains $\tilde U_\pm$ are flowers on the intervals $L_{n\pm}$, hence, according to Lemma~\ref{L_n_combinatorial_lemma}, we have
	$$
	\tilde U_\pm\Subset p\cR^nf_\pm(\tilde U_\pm).
	$$
	Combining this with Lemma~\ref{Interval_not_too_short_lemma}, we conclude that 
	$$
	\mod(\tl U_\pm,p\cR^nf_\pm(\tilde U_\pm))>\nu,
	$$
	possibly, after decreasing the constant $\nu(\delta,\Delta,\Theta)$.
	
	Finally, let domains $D$, $U_+$ and $U_-$ be affine rescalings of $\tl D$, $\tl U_+$ and $\tl U_-$ respectively, rescaled by the affine map that takes $C_n$ to $[0,1]$. According to our construction, the renormalization $\cR^nf$ extends to a power-like Lorenz map $\cR^nf\colon U_\pm\to D$ that satisfies condition~\ref{H_nu_item1} of Definition~\ref{H_nu_def}. Condition~\ref{H_nu_item4} of Definition~\ref{H_nu_def} is satisfied due to Lemma~\ref{pRnf_distorted_root_lemma}, possibly, after decreasing the constant $\nu$. According to Definition~\ref{real_bounds_def}, the intervals $C_{n+}$ and $C_{n-}$ are commensurable, hence, after 
	possibly decreasing the constant $\nu$ again, the map $\cR^nf$ is guaranteed to satisfy condition~\ref{H_nu_item3} of Definition~\ref{H_nu_def}. As the last step, we verify that the second part of inequality~(\ref{mu1mu2_inequality}) implies condition~\ref{H_nu_item2} of Definition~\ref{H_nu_def}, after possibly decreasing the constant $\nu$ again. This completes the proof of Theorem~\ref{main_theorem}.
\end{proof}

\subsection{Analyticity of renormalization}
In this subsection we give a proof of Theorem~\ref{main_theorem2}.

For a Jordan domain $\Omega\subset\bbC$, let $\mathcal B(\Omega)$ denote the space of all analytic maps $g\colon\Omega\to\bbC$ that continuously extend to the closure $\overline\Omega$. The set $\mathcal B(\Omega)$ equipped with the sup-norm, is a complex Banach space.
If $\Omega$ is symmetric with respect to the real axis, we let $\mathcal B^\bbR(\Omega)\subset\mathcal B(\Omega)$ denote the real Banach space of all real-symmetric functions from $\mathcal B(\Omega)$.

Given a positive real number $\alpha>1$, the function
$$
p_{\alpha+}\colon\bbC\setminus\bbR^-\to\bbC
$$
is defined as the branch of the map $z\mapsto z^\alpha$ which maps positive reals to positive reals. Similarly we define the function
$$
p_{\alpha-}\colon\bbC\setminus\bbR^+\to\bbC
$$
as the branch of the map $z\mapsto -(-z)^\alpha$ which maps negative reals to negative reals.

For $c\in\bbC\setminus\{0,1\}$, let $\phi_{c+},\phi_{c-}\colon \bbC\to\bbC$ be the affine maps such that $\phi_{c-}([0,c])=[-1,0]$ and $\phi_{c+}([c,1])=[0,1]$, where $[a,b]$ denotes the straight line segment between two complex numbers $a,b\in\bbC$.

For a compact set $K\subset\bbC$ and a positive real number $r>0$, let $N_r(K)$ denote the $r$-neighborhood of $K$ in $\bbC$, namely,
$$
N_r(K)=\{z\in\bbC\mid \min_{w\in K}|z-w|<r\}.
$$

\begin{definition}
	For a positive real number $s>0$, let $\mathbf B_{s-}\subset\mathcal B(N_s([-1,0]))$ be the set of all maps $\psi_-\in \mathcal B(N_s([-1,0]))$ that are univalent in some neighborhood of the interval $[-1,0]$, and such that $\psi_-(-1)=0$ and $0<\Re(\psi_-(0))<1$. Similarly, let $\mathbf B_{s+}\subset\mathcal B(N_s([0,1]))$ be the set of all maps $\psi_+\in \mathcal B(N_s([0,1]))$ that are univalent in some neighborhood of the interval $[0,1]$, and such that $\psi_+(1)=1$ and $0<\Re(\psi_+(0))<1$.
\end{definition}

\begin{proposition}
	For any real number $s>0$, the sets $\mathbf B_{s-}$ and $\mathbf B_{s+}$ are real-symmetric codimension~$1$ affine submanifolds of $\mathcal B(N_s([-1,0]))$ and $\mathcal B(N_s([0,1]))$ respectively.
\end{proposition}
\begin{proof}
	Let $\mathcal B_{s,-1}$ 
	denote the Banach subspace of $\mathcal B(N_s([-1,0]))$ that consists of all $\psi\in\mathcal B(N_s([-1,0]))$, such that $\psi(-1)=0$.
	Then, $\mathbf B_{s-}$ is an open subset of the Banach space $\mathcal B_{s,-1}$. Real symmetry of $\mathbf B_{s-}$ follows from the construction. The proof for $\mathbf B_{s+}$ is similar.
\end{proof}

\begin{definition}
	For a positive real number $s>0$, let $\mathbf A_s$ be the set of all pairs of maps $f=(f_-,f_+)$, such that
	\begin{equation}\label{A_s_representation_eq}
	f_\pm=\psi_\pm\circ p_{\alpha\pm}\circ\phi_{c\pm},
	\end{equation}
	for some $c\in\bbC$ satisfying $0<\Re c<1$, $\phi_{c-}\in \mathbf B_{s-}$ and $\phi_{c+}\in \mathbf B_{s+}$.
\end{definition}

The set $\mathbf A_s$ has a natural structure of a real-symmetric Banach manifold, obtained as a direct product $\{c\in\bbC\mid 0<\Re c<1\}\times\mathbf B_{s-}\times\mathbf B_{s+}$. It is clear from the construction that all elements of the real slice $\mathbf A_s^\bbR$ are analytic Lorenz maps.

Theorem~\ref{main_theorem2} is a direct corollary of the following theorem:

\begin{theorem}\label{main_theorem2a}
	For any pair of real numbers $\delta,\Delta>0$ and a finite set $\Theta\subset\mathfrak P$, there exists a positive real number $s=s(\delta,\Delta,\Theta)>0$, such that Theorem~\ref{main_theorem2} holds for $\mathbf M=\mathbf A_s$.
\end{theorem}

\begin{proof}
Let $\nu=\nu(\delta,\Delta,\Theta)$ be the same as in Theorem~\ref{main_theorem}. Let $\mathcal A\subset\mathfrak{L}_{\nu}$ be the set of all Lorenz maps $f\in\mathfrak L_\nu$ with real $(\delta,\Delta)$-bounds of level $0$. Each such function $f$ can be represented in the form~(\ref{A_s_representation_eq}). It follows from Lemma~\ref{derivative_bounds_lemma} that the derivatives $\psi_\pm'$ are bounded from above and away from zero uniformly in $f\in \mathcal A$. Now Koebe $1/4$-Theorem implies existence of a positive real number $s_0=s_0(\delta,\Delta,\nu)>0$, such that appropriate restrictions of every map $f\in \mathcal A$ belong to $\mathbf A_{s_0}$ and the corresponding maps $\psi_-$ and $\psi_+$ are univalent on their domains $N_{s_0}([-1,0])$ and $N_{s_0}([0,1])$ respectively. According to Montel's Theorem, the family $\mathcal A$ is relatively compact in $\mathbf A_{s_0}$. Let $\overline{\mathcal A}\subset\mathbf A_{s_0}$ denote the closure of $\mathcal A$. We note that uniform limits of bounded univalent functions are either constant or univalent, and real $(\delta,\Delta)$-bounds of level $0$ for maps from $\mathcal A$ exclude the first possibility. Hence, for every map $f\in\overline{\mathcal A}\subset\mathbf A_{s_0}$, the corresponding maps $\psi_-$ and $\psi_+$ are univalent on their domains $N_{s_0}([-1,0])$ and $N_{s_0}([0,1])$ respectively. Together with compactness of $\overline{\mathcal A}$, this implies existence of a real number $r>0$, such that a sufficiently small open neighborhood of $\overline{\mathcal A}$ in $\mathbf A_{s_0/2}$ is contained in $\mathfrak L_r$. 
Fix $N=n_0(r,\delta,\Delta,\Theta)$, where $n_0$ is the same as in Theorem~\ref{main_theorem} and let $\mathcal K\subset\mathcal A\cap\mathcal S_\Theta^{N+1}$ be the set of those Lorenz maps that have real $(\delta,\Delta)$-bounds of level at least $N$. It follows from Theorem~\ref{main_theorem} that
$$
\cR^N(\mathcal K)\subset\mathcal A\subset \mathbf A_{s_0}.
$$
By continuity of renormalization for maps of bounded type and with real bounds, it follows that there exists an open set $\mathcal O\subset\mathbf A_{s_0/2}$, such that $\mathcal K\subset\mathcal O$ and $\cR^N(\mathcal O)\subset\mathbf A_{s_0/2}$. By our construction, property~(\ref{Analyt_prop1}) of Theorem~\ref{main_theorem2} holds for $\mathbf M=\mathbf A_s$, where $s:=s_0/2$. Property~(\ref{Analyt_prop2}) of Theorem~\ref{main_theorem2} follows from Theorem~\ref{main_theorem}.
\end{proof}

\bibliographystyle{amsalpha}
\bibliography{biblio}
\end{document}